\documentclass[10pt,onepage,reqno]{amsart}

\usepackage[utf8]{inputenc}
\usepackage{hyperref}
\usepackage{amsmath,array,color}
\usepackage{amssymb,tabularx}
\usepackage{listings}
\usepackage{fancyvrb}
\usepackage{cancel}
\usepackage{multirow}
\usepackage{tocvsec2}
\usepackage{booktabs}
\usepackage{float}

\usepackage{tikz}
\usetikzlibrary{positioning,arrows,shapes,decorations.pathmorphing,shadows}

\tikzstyle{every picture}+=[remember picture]
\tikzstyle{colore1}=[ball color=SandyBrown] 
\tikzstyle{colore2}=[ball color=White] 
\tikzstyle{coloreno}=[ball color=Red] 
\tikzstyle{colore3}=[ball color=LightGrey] 
\tikzstyle{colore4}=[ball color=LightGrey] 
\tikzstyle{freccia}=[thick] 

\usepackage[a4paper,lmargin=2.0cm,rmargin=2.0cm,tmargin=2.0cm,bmargin=2.0cm]{geometry}

\usepackage[backend=biber,style=alphabetic,doi=false,isbn=false,url=false,urldate=long,sorting=nty]{biblatex}
\AtEveryBibitem{\clearlist{language}}
\usepackage{environ}
\NewEnviron{myequation}{%
\begin{equation}
\scalebox{0.88}{$\BODY$}
\end{equation}}
\NewEnviron{myequation1}{%
\begin{equation}
\scalebox{0.88}{$\BODY$}
\end{equation}}
\NewEnviron{myequation2}{%
\begin{equation}
\scalebox{0.7}{$\BODY$}
\end{equation}}

\newcommand{\LF}{{\mathcal L}}

\newcommand{\shortform}[1]{\ensuremath{{\scriptstyle\boldsymbol{#1}}}}

\newcommand{\Id}{\mathop{\mathrm{Id}}}
\newcommand{\Cl}[1]{\mathrm{Cl}(#1)}

\newcommand{\Spin}[1]{\ensuremath{\text{\upshape\rmfamily Spin}(#1)}}

\newcommand{\SO}[1]{\ensuremath{\text{\upshape\rmfamily SO}(#1)}}
\newcommand{\U}[1]{\ensuremath{\text{\upshape\rmfamily U}(#1)}}
\newcommand{\Gtwo}{\ensuremath{\text{\upshape\rmfamily G}_2}}
\newcommand{\SU}[1]{\ensuremath{\text{\upshape\rmfamily SU}(#1)}}

\newcommand{\End}[1]{\ensuremath{\text{\upshape\rmfamily End}(#1)}}
\newcommand{\Sp}[1]{\ensuremath{\text{\upshape\rmfamily Sp}(#1)}}

\newcommand{\ug}{\;\shortstack{{\tiny\upshape def}\\=}\;}
\newcommand{\snform}{\ensuremath{\Phi}}
\newcommand{\spinform}[1]{\ensuremath{\Phi_{\Spin{#1}}}}

\newcommand{\liespin}[1]{\mathop{\mathfrak{spin}}(#1)}
\newcommand{\liesp}[1]{\mathop{\mathfrak{sp}}(#1)}
\newcommand{\lieso}[1]{\mathop{\mathfrak{so}}(#1)}
\newcommand{\lie}[1]{\mathop{\mathfrak{#1}}}
\newcommand{\lieu}[1]{\mathop{\mathfrak{u}}(#1)}
\newcommand{\Ffour}{\mathrm{F}_4}

\newcommand{\Esix}{\mathrm{E}_6}
\newcommand{\Eseven}{\mathrm{E}_7}
\newcommand{\Eeight}{\mathrm{E}_8}
\newcommand{\EI}{\mathrm{E\,I}}
\newcommand{\EII}{\mathrm{E\,II}}
\newcommand{\EIII}{\mathrm{E\,III}}
\newcommand{\EIV}{\mathrm{E\,IV}}
\newcommand{\EV}{\mathrm{E\,V}}
\newcommand{\EVI}{\mathrm{E\,VI}}
\newcommand{\EVII}{\mathrm{E\,VII}}
\newcommand{\EVIII}{\mathrm{E\,VIII}}
\newcommand{\EIX}{\mathrm{E\,IX}}
\newcommand{\FI}{\mathrm{F\,I}}
\newcommand{\FII}{\mathrm{F\,II}}
\newcommand{\GI}{\mathrm{G\,I}}

\newcommand{\blocktext}{\mathrm{block}}

\newcommand{\coniugiobaseraw}{\mathrm{D}}
\newcommand{\coniugiobase}[1]{\coniugiobaseraw_{#1}}

\newcommand{\form}[1]{\ensuremath{\snform_{#1}}}

\newcommand{\norm}[1]{\Vert #1\Vert}

\newcommand{\CC}{\mathbb{C}}   
\newcommand{\HH}{\mathbb{H}}   
\newcommand{\RR}{\mathbb{R}}   
\newcommand{\ZZ}{\mathbb{Z}}

\newcommand{\OO}{\mathbb{O}}
\newcommand{\Se}{\mathbb{S}}
\newcommand{\CP}[1]{\mathbb{C}P^{#1}}
\newcommand{\OP}[1]{\mathbb{O}P^{#1}}
\newcommand{\OH}[1]{\mathbb{O}H^{#1}}
\newcommand{\HP}[1]{\mathbb{H}P^{#1}}

\newcommand{\I}{\mathcal{I}} 
\newcommand{\J}{J} 

\numberwithin{equation}{section}

\newtheorem{Theorem}{Theorem}[section]
\newtheorem*{te*}{Theorem}

\newtheorem{Proposition}[Theorem]{Proposition}
\newtheorem{Corollary}[Theorem]{Corollary}

\theoremstyle{definition}
\newtheorem{Definition}[Theorem]{Definition}

\theoremstyle{remark}
\newtheorem{Remark}[Theorem]{Remark}

\newtheorem{Example}[Theorem]{Examples}

\begin{document}

\title{The Role of $\Spin{9}$ in Octonionic Geometry}

\subjclass[2010]{Primary 53C26, 53C27, 53C35, 57R25.}
\keywords{$\Spin{9}$; octonions; vector fields on spheres; Hopf fibration; locally conformally parallel; Clifford structure; Clifford system; symmetric spaces.}

\author{Maurizio Parton}
\address{Dipartimento di Economia, Universit\`a di Chieti-Pescara, Viale della Pineta 4, I-65129 Pescara, Italy}
\email{\href{mailto:parton@unich.it}{parton@unich.it}}

\author{Paolo Piccinni}
\address{Dipartimento di Matematica, Sapienza---Universit\`a di Roma, Piazzale Aldo Moro 2, I-00185 Roma, Italy}
\email{\href{mailto:piccinni@mat.uniroma1.it}{piccinni@mat.uniroma1.it}}

\thanks{The authors were supported by the group GNSAGA of INdAM and by the PRIN Project of MIUR ``Variet\`a reali e complesse: geometria, topologia e analisi armonica''. M.~P. was also supported by Universit\`a di Chieti-Pescara, Dipartimento di Economia. P.~P. was also supported by Sapienza Universit\`a di Roma Project ``Polynomial identities and combinatorial methods in algebraic and geometric structures''.}

\dedicatory{Dedicated to the memory of Thomas Friedrich}

\begin{abstract}
Starting from the 2001 Thomas Friedrich's work on $\Spin{9}$, we review some interactions between $\Spin{9}$ and geometries related to octonions. Several topics are discussed in this respect: explicit descriptions of the $\Spin{9}$ canonical 8-form and its analogies with quaternionic geometry as well as the role of $\Spin{9}$ both in the classical problems of vector fields on spheres and in the geometry of the octonionic Hopf fibration. Next, we deal with locally conformally parallel $\Spin{9}$ manifolds in the framework of intrinsic torsion. Finally, we discuss applications of Clifford systems and Clifford structures to Cayley-Rosenfeld planes and to three series of Grassmannians.
\end{abstract}

\maketitle

\tableofcontents

\section{Introduction}\label{sec:intro}

One of the oldest evidences of interest for the $\Spin{9}$ group in geometry goes back to the 1943~Annals of Mathematics paper by D. Montgomery and H. Samelson \cite{MoSTGS}, which classifies compact Lie groups that act transitively and effectively on spheres, and gives the following list: $$\SO{n}, \; \U{n}, \; \SU{n}, \; \Sp{n}, \; \Sp{n}\cdot \U{1}, \; \Sp{n}\cdot \Sp{1}, \; \Gtwo, \; \Spin{7}, \; \Spin{9}.$$ 

In particular, $\Spin{9}$ acts transitively on the sphere $S^{15}$ through its Spin representation, and the stabilizer of the action is a subgroup $\Spin{7}$.

In the following decade, the above groups, with the only exception of $\Sp{n} \cdot \U{1}$, appeared in the celebrated M. Berger theorem \cite{BerGHH} as the list of the possible holonomy groups of irreducible, simply~connected, and non symmetric Riemannian manifolds. Next, in the decade after that, 
D.~Alekseevsky~\cite{AleRSE} proved that $\Spin{9}$ is the Riemannian holonomy of only symmetric spaces, namely of the Cayley projective plane and its non-compact dual. Accordingly, $\Spin{9}$ started to be omitted in the Berger theorem statement. Much later, a geometric proof of Berger theorem was given by C.~Olmos~\cite{OlmGPB}, using submanifold geometry of orbits and still referring to possible transitive actions on spheres.

Moreover, in the last decades of the twentieth century, compact examples have been shown to exist for almost all classes of Riemannian manifolds related to the other holonomy groups in the Berger list. References for this are the books by S. Salamon and D. Joyce \cite{SalRGH,JoyCMS,JoyRHG}. For these reasons, around~the year 2000, the best known feature of $\Spin{9}$ seemed to be that it was a group that had been removed from an interesting list.

Coming into the new millennium, since its very beginning, new interest in dealing with different aspects of octonionic geometry appeared, and new features of structures and weakened holonomies related to $\Spin{9}$ were pointed out. 
Among the references, there is, notably, the J. Baez extensive Bulletin AMS paper on octonions \cite{BaeOct} as well as the not less extensive discussions on his webpage~\cite{BaeTWF}. Next, and from a more specific point of view, there is the Thomas Friedrich paper on ``weak~$\Spin{9}$-structures'' \cite{FriWSS}, which proposes a way of dealing with a $\Spin{9}$ structure, and this was later recognized by A.~Moroianu and U. Semmelmann \cite{MoSCSR} to fit in the broader context of Clifford structures. Also, the M. Atiyah and J. Berndt paper in Surveys in Differential Geometry \cite{AtBPPS} shows interesting connections with classical algebraic geometry. Coming to very recent contributions, it~is worth mentioning the work by N.~Hitchin \cite{HitSOO} based on a talk for R. Penrose's 80th birthday, which deals with $\Spin{9}$ in relation to further groups of interest in octonionic geometry.

The aim of the present article is to give a survey of our recent work on $\Spin{9}$ and octonionic geometry, in part also with L. Ornea and V. Vuletescu, and mostly contained in the references \cite{PaPSAC,PaPSMS,OPPSGO,PaPECS,PPVCSO,PicCSE,PicSGC}. 

Our initial motivation was to give a construction, as simple as possible, of the canonical octonionic 8-form $\spinform{9}$ that had been defined independently through different integrals  by M. Berger \cite{BerCCP} and by R. Brown and A.~Gray \cite{BrGRMH}. Our construction of $\spinform{9}$ uses the already mentioned definition of a~$\Spin{9}$-structure proposed by Thomas Friedrich and has a strong analogy with the construction of a~$\Sp{2}\cdot \Sp{1}$-structure in dimension 8 (see Section \ref{Canonical differential forms} as well as \cite{PaPSAC}). By developing our construction of $\spinform{9}$, we realized that some features of the $S^{15}$ sphere can be conveniently described through the same approach that we used. The fact that $S^{15}$ is the lowest dimensional sphere that admits more than seven global linearly independent tangent vector fields is certainly related to the Friedrich point of view. Namely, by developing a convenient linear algebra, 
we were able to prove that the full system of maximal linearly independent vector fields on any $S^n$ sphere can be written in terms of the unit imaginary elements in $\CC,\HH,\OO$ and the complex structures that Friedrich associates with $\Spin{9}$ (see Section~\ref{Spheres} and \cite{PaPSMS}). Another feature of $S^{15}$ is, of course, that it represents the total space of the octonionic Hopf fibration, whose group of symmetries is $\Spin{9} \subset \SO{16}$. Here, the Friedrich approach to $\Spin{9}$ allows to recognize both the non-existence of nowhere zero vertical vector fields and some simple properties of locally conformally parallel $\Spin{9} $-structures (here, see~Theorem~\ref{te:A}, Section~\ref{se:lcp}, and \cite{OPPSGO}). We then discuss the broader contexts of Clifford structures and Clifford systems, that allow us to deal with the complex Cayley projective plane, whose geometry and topology can be studied by referring to its projective algebraic model, known as the fourth Severi variety. With similar methods, one can also study the structure and properties of the remaining two Cayley-Rosenfeld projective planes (for all of this, see Sections \ref{Clifford}--\ref{Rosenfeld}, and \cite{PaPECS,PPVCSO}). Finally, Clifford structures and Clifford systems can be studied in relation with the exceptional symmetrical spaces of compact type as well as with some real, complex, and quaternionic Grassmannians that carry a~geometry very much related to octonions (see Sections \ref{Exceptional} and \ref{Grassmannians}, and 
\cite{PicCSE,PicSGC}).

During the years of our work, we convinced ourselves that $\Spin{9}$ influences not only 16-dimensional Riemannian geometry, but also aspects related to octonions of some lower dimensional and higher dimensional geometry. It is, in fact, our hope that the reader of this survey can share the feeling of the beauty of $\Spin{9}$, that seems to have some role in geometry, besides being a~group that had been removed from an interesting list.

\section{Preliminaries, Hopf Fibrations, and Friedrich's work}\label{Preliminaries}

The multiplication involved in the algebra ($\OO$) of octonions can be defined from the one in quaternions ($\HH$) by the Cayley--Dickson process: if $x=h_1+h_2 e$, $x'=h'_1+h'_2e \in \OO$, then
\begin{equation}\label{oct}
xx'=(h_1h'_1-\overline h'_2 h_2) + (h_2 \overline h'_1 + h'_2 h_1)e,
\end{equation}
where $\overline h'_1, \overline h'_2$ are the conjugates of $h'_1, h'_2 \in \HH$. As for quaternions, the conjugation $\overline{x}=\overline h_1-h_2e$ 
is related to the non-commutativity: $\overline{x x'}=\overline x'\overline x.$ The associator 
\[
[x,x',x''] = (xx')x''-x(x'x'')
\] 
vanishes whenever two among $x,x',x'' \in \OO$ are equal or conjugate. For a survey on octonions and their applications in geometry, topology, and mathematical physics, the excellent article ~\cite{BaeOct} by J. Baez is a basic reference.

The 16-dimensional real vector space $\OO^2$ decomposes into its \emph{octonionic lines,}  
\[
l_m\ug\{(x,mx)\vert x\in\OO\} \quad \text{or} \quad l_\infty\ug\{(0,x)\vert x\in\OO \},
\] 
that intersect each other only at the origin $(0,0) \in \OO^2$. Here, $m \in S^8 =  \OP{1} = \OO \cup \{\infty\}$ parametrizes the set of octonionic lines ($l$), whose volume elements ($\nu_{l} \in \Lambda^8 l$) allow the following \emph{canonical 8-form} on $\OO^2 = \RR^{16}$ to be defined:
\begin{equation}\label{8form}
\spinform{9} =\frac{110880}{\pi^4} \int_{\OP{1}}p_l^*\nu_l\,dl\in\Lambda^8(\RR^{16}), 
\end{equation}
where $p_l$ denotes the orthogonal projection $\OO^2 \rightarrow l$. 

The definition of $\spinform{9}$ through this integral was given by M. Berger \cite{BerCCP}, and here we chose the proportionality factor in such a way to make integers 
and with no common factors  the coefficients of $\spinform{9}$
as exterior $8$-form in $\RR^{16}$. 
The notation is motivated by the following:

\begin{Proposition}[\cite{CorASH}]\label{de:spin9} 
The subgroup of $\mathrm{GL}(16, \RR)$ preserving $\spinform{9}$ is the image of $\Spin{9}$ under its spin representation into $\RR^{16}$.
\end{Proposition}

Thus, $\Spin{9} \subset \SO{16}$, so that 16-dimensional oriented Riemannian manifolds are the natural setting for $\Spin{9}$-structures.
The following definition was proposed by Th. Friedrich, \cite{FriWSS}.

\begin{Definition}\label{de:spin9structure}
Let $(M,g)$ be a 16-dimensional oriented Riemannian manifold. A $\Spin{9}$ \emph{structure} on $M$ is the datum of any of the following equivalent alternatives.
\begin{enumerate}
\item A rank $9$ subbundle, $E=E^9\subset\End{TM}$, locally spanned by endomorphisms $\{\I_\alpha\}_{\alpha=1,\dots 9}$ with
\begin{equation}\label{top}
\I^2_\alpha = \Id, \qquad \I^*_\alpha = \I_\alpha, \quad \text{and} \quad \I_\alpha \I_\beta = - \I_\beta \I_\alpha \quad\text{for}\quad \alpha \neq \beta,
\end{equation}
where $\I^*_\alpha$ denotes the adjoint of $\I_\alpha$.
\item A reduction, $\mathcal{R}$, of the principal bundle, $\mathcal F (M)$, of orthonormal frames from $\SO{16}$ to $\Spin{9}$.
\end{enumerate}
\end{Definition}

In particular, the existence of a $\Spin{9}$ structure depends only on the conformal class of the metric $g$ on $M$.

We now describe the vector bundle $E^9$ when $M$ is the model space ($\RR^{16}$). Here, $\I_1, \dots, \I_9$ can be chosen as generators of the Clifford algebra ($\Cl{9}$), the endomorphisms's algebra of its $16$-dimensional real representation, $\Delta_9 = \RR^{16} = \OO^2$.
Accordingly, unit vectors ($v \in S^8 \subset \RR^9$) can be viewed via the Clifford multiplication as symmetric endomorphisms: $v: \Delta_9 \rightarrow \Delta_9$. 

The explicit way to describe this action is by $v = u + r \in S^8$ ($u \in \OO$, $r \in \RR$, $u\overline u + r^2 =1$), acting on pairs $(x,x') \in \OO^2$:
\begin{equation}\label{HarSpC}
\left(
\begin{array}{c}
x \\
x'
\end{array}
\right)
\longrightarrow
\left(
\begin{array}{cc}
r & R_{\overline u} \\ 
R_u & -r
\end{array}
\right)  
\left(
\begin{array}{c} 
x \\ 
x'
\end{array}
\right),
\end{equation}
where $R_u, R_{\overline u}$ denotes the right multiplications by $u, \overline u$, respectively  (cf.~\cite{HarSpC} (p. 288)).

A basis of the standard $\Spin{9}$ structure on $\OO^2 = \RR^{16}$ can be written by looking at action~\eqref{HarSpC} and at the following nine vectors: 
\[
(0,1),(0,i),(0,j),(0,k),(0,e),(0,f),(0,g),(0,h)\quad\text{and}\quad(1,0) \in S^8 \subset \OO \times \RR = \RR^9,
\]
where $f=ie$, $g=je$, and $h=ke$, and their products are ruled by \eqref{oct}.
This gives the following symmetric endomorphisms:
\begin{equation}\label{eq:IO}
\begin{aligned}
\I_1&=\left(
\begin{array}{c|c}
0 & \Id \\ \hline
\Id & 0
\end{array}
\right)\enspace,\qquad &
\I_2&=\left(
\begin{array}{c|c}
0 & -R_i \\ \hline
R_i & 0
\end{array}
\right)\enspace,\qquad &
\I_3&=\left(
\begin{array}{c|c}
0 & -R_j \\ \hline
R_j & 0
\end{array}
\right)\enspace, \\
\I_4&=\left(
\begin{array}{c|c}
0 & -R_k \\ \hline
R_k & 0
\end{array}
\right)\enspace,\qquad &
\I_5&=\left(
\begin{array}{c|c}
0 & -R_e \\ \hline
R_e & 0
\end{array}
\right)\enspace,\qquad &
\I_6&=\left(
\begin{array}{c|c}
0 & -R_f\\ \hline
R_f & 0
\end{array}
\right)\enspace, \\
\I_7&=\left(
\begin{array}{c|c}
0 & -R_g \\ \hline
R_g & 0
\end{array}
\right)\enspace,\qquad &
\I_8&=\left(
\begin{array}{c|c}
0 & -R_h \\ \hline
R_h & 0
\end{array}
\right)\enspace,\qquad &
\I_9&=\left(
\begin{array}{c|c}
\Id & 0 \\ \hline
0 & -\Id
\end{array}
\right)\enspace,
\end{aligned}
\end{equation}
where $R_i,\dots,R_h$ are the right multiplications by the 7 unit octonions, $i,\dots,h$. Their spanned subspace
\begin{equation}\label{eq:V9}
E^9\ug<\I_1,\dots,\I_9>\subset\End{\RR^{16}}
\end{equation}
is such that the following proposition applies.

\begin{Proposition}[\cite{CorASH}]
The subgroup of $\SO{16}$ preserving $E^9$ is $\Spin{9}.$
\end{Proposition}

\tikzstyle{format}=[draw=none]
The projection $\OO^2 -\{0\} \rightarrow \OP{1}$ associated with decomposition into the octonionic lines $l_m$, $l_\infty$ is a non-compact version of the octonionic Hopf fibration:
\[
S^{15} \rightarrow \OP{1} \cong S^8,
\]
that is the unique surviving possibility when passing from quaternions to octonions from the series of quaternionic Hopf fibrations:
\[
S^{4n+3} \rightarrow \HH P^n.
\]

Recall that the latter enter into Figure~\ref{fig:proto} that encodes prototypes of several structures of interest in quaternionic geometry.  

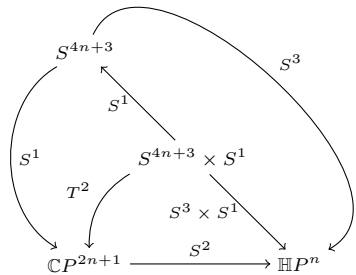
\begin{figure}[H]
{\scriptsize
\begin{center}

\begin{tikzpicture}[node distance=2cm,auto]

\draw node[] (S4n+3) {$S^{4n+3}\times S^1$};
	\draw[->] node[format, below left of=S4n+3] (CP2n+1) {$\CP{2n+1}$}
	(S4n+3) edge [bend right = 30] node[swap] {\tiny{$T^2$}} (CP2n+1);
\draw[->] node[format, below right of=S4n+3] (HPn) {$\HP{n}$}
	(S4n+3) edge [left = 30] node {\tiny{$S^3\times S^1$}} (HPn);
\draw[->] (CP2n+1) edge node {\tiny{$S^2$}} (HPn);
\draw[->] node[format, above left of=S4n+3] (S) {$S^{4n+3}$}
	(S4n+3) edge [left = 30] node {\tiny{$S^1$}} (S);
\draw[->] (S) edge [bend right = 60] node {\tiny{$S^1$}} (CP2n+1);
\draw[->] (S) edge [bend left = 108] node {\tiny{$S^3$}} (HPn);

\end{tikzpicture}
\end{center}
}\caption{The prototype of foliations related to locally conformally hyperk\"ahler manifolds.}\label{fig:proto}
\end{figure}


At the center of the diagram, there is the locally conformally hyperk\"ahler Hopf manifold $S^{4n+3}\times S^1$. All the other manifolds are leaf spaces of foliations on them, such as the 3-Sasakian sphere $S^{4n+3}$, the~positive K\"ahler-Einstein twistor space $\CP{2n+1}$, and the positive quaternion K\"ahler $\HP{n}$. Most of the foliations carry similar structures on their leaves, for example, one has locally conformally hyperk\"ahler Hopf surfaces of $S^3\times S^1$. This prototype diagram is only an example, since, \mbox{when a compact} locally hyperk\"ahler manifold has compact leaves on the four canonically defined vertical foliations, our~diagram still makes sense, albeit in the broader orbifold category (cf. \cite{OrPLCK}).

When $n=3$, there are also the octonionic Hopf fibrations, as in Figure~\ref{fig:oct}: 

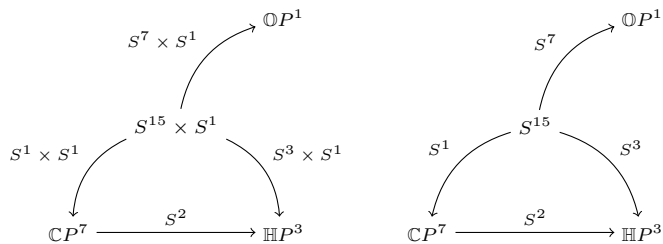
\begin{figure}[H] 
{\scriptsize
\begin{center}
\begin{tikzpicture}[node distance=2cm,auto]

\draw node[] (S15) {$S^{15}\times S^1$};
	\draw[->] node[format, below left of=S15] (CP7) {$\CP{7}$}
	(S15) edge [bend right = 30] node[swap] {\tiny{$S^1\times S^1$}} (CP7);
\draw[->] node[format, below right of=S15] (HP3) {$\HP{3}$}
	(S15) edge [bend left = 30] node {\tiny{$S^3\times S^1$}} (HP3);
\draw[->] (CP7) edge node {\tiny{$S^2$}} (HP3);
\draw[->] node[format, above right of=S15] (OP1) {$\OP{1}$}
	(S15) edge [bend left = 30] node {\tiny{$S^7\times S^1$}} (OP1);
	
\end{tikzpicture}
\qquad 
\begin{tikzpicture}[node distance=2cm,auto]

\draw node[] (S15) {$S^{15}$};
	\draw[->] node[format, below left of=S15] (CP7) {$\CP{7}$}
	(S15) edge [bend right = 30] node[swap] {\tiny{$S^1$}} (CP7);
\draw[->] node[format, below right of=S15] (HP3) {$\HP{3}$}
	(S15) edge [bend left = 30] node {\tiny{$S^3$}} (HP3);
\draw[->] (CP7) edge node {\tiny{$S^2$}} (HP3);
\draw[->] node[format, above right of=S15] (OP1) {$\OP{1}$}
	(S15) edge [bend left = 30] node {\tiny{$S^7$}} (OP1);
\end{tikzpicture}
\end{center}
}
\caption{In the octonionic case, an additional fibration appears.}\label{fig:oct}
\end{figure}

These have no arrow connecting $\OP{1}$ with $\HP{3}$ and $\CP{7}$, since the complex and quaternionic Hopf fibrations are not subfibrations of the octonionic one (cf. \cite{LoVHFO} as well as Theorem \ref{te:A} in the following Section \ref{Hopf}).

Coming back to $\Spin{9}$, as the title of Th. Friedrich's article \cite{FriWSS} suggests, there is a scheme for ``weak $\Spin{9}$ structures'' to include some possibilities besides the very restrictive holonomy $\Spin{9}$ condition. Although the original A. Gray proposal \cite{GraWHG} to look at ``weak holonomies'' was much later shown by B. Alexandrov \cite{AleWeH} not to produce new geometries for the series of groups quoted in the Introduction, one can refer to the symmetries of relevant tensors to understand the possibilities for any $G$-structure. We briefly recall a unified scheme that one can refer to, following the presentation of Ref.~\cite{AgrSLN}.

By definition, a $G$-structure on an oriented Riemannian manifold ($M^n$) is a reduction ($\mathcal R \subset \mathcal F(M^n)$) of the orthonormal frame bundle to the subgroup $G \subset \SO{n}$. The Levi Civita connection ($Z$),  thought~of as a 1-form on $\mathcal F (M^n)$ with values in the Lie algebra $\lieso{n}$, restricts to a connection on $\mathcal R$, decomposing with respect to the Lie algebra splitting $\lieso{n}= \mathfrak g \oplus \mathfrak m$, as 
\[
Z\vert_{T(\mathcal R)} = Z^* \oplus \Gamma. 
\]

Here, $Z^*$ is a connection in the principal $G$-bundle $\mathcal R$, and $\Gamma$ is a 1-form on $M^n$ with values in the associated bundle $\mathcal R \times_G \mathfrak m$, called the \emph{intrinsic torsion} of the $G$-structure. Of course, the condition $\Gamma =0$ is equivalent to the inclusion $\text{Hol} \subset G$ for the Riemannian holonomy, and $G$-structures with $\Gamma \neq 0$ are called non-integrable. 

This scheme can be used, in particular, when $G$ is the stabilizer of some tensor $\eta$ in $\mathbb R^n$, so that the $G$-structure on $M$ defines a global tensor $\eta$, and here, $\Gamma = \nabla \eta$ can be conveniently thought of as a~
section of the vector bundle:
\[
\mathcal W = T^* \otimes \mathfrak m.
\]

Accordingly, the action of $G$ splits $\mathcal W$ into irreducible $G$-components: $\mathcal W = \mathcal W_1 \oplus \dots \oplus\mathcal W_k$. 

A prototype of such decompositions occurs when $G=\U{n} \subset \SO{2n}$ and yields the four irreducible components of the so-called Gray--Hervella classification when $n \geq 3$ \cite{GrHSCA}. It is a fact from the representation theory that there are several further interesting cases that yield four irreducible components. This occurs when $G=\Gtwo \subset \SO{7}, G= \Sp{2} \cdot \Sp{1} \subset \SO{8}, G=\Spin{9} \subset \SO{16}$, as computed in Refs. \cite{FeGRMS}, \cite{SwaHQK} (p. 115) and \cite{FriWSS}, respectively:
\begin{equation}\label{W}
\mathcal W = \mathcal W_1 \oplus \mathcal W_2 \oplus \mathcal W_3  \oplus \mathcal W_4.
\end{equation}

For the mentioned four situations, the last component, $\mathcal W_4$, is the ``vectorial type'' one, and gives rise to a 1-form $\theta$ on the manifold. A general theory of $G$-structures in this last class, $\mathcal W_4$, of typically locally conformally parallel $G$-structures, was developed in Ref. \cite{AgFGSV}.  In Section \ref{se:lcp}, we revisit $\mathcal W_4$ in the $G=\Spin{9}$ case, following Ref. \cite{FriWSS} as well as our previous work \cite{OPPSGO}. 

\section{The canonical 8-form $\spinform{9}$}\label{Canonical differential forms}

The $\Lambda^2\RR^{16}$ space of $2$-forms in $\RR^{16}$ decomposes under $\Spin{9}$ as
\begin{equation}\label{decomposition}
\Lambda^2\RR^{16} = \Lambda^2_{36} \oplus \Lambda^2_{84}
\end{equation}
(cf.~\cite{FriWSS} (p. 146)), where  $\Lambda^2_{36} \cong \liespin{9}$ and $ \Lambda^2_{84} =\mathfrak{m}$ is an orthogonal complement in $\Lambda^2 \cong \lieso{16}$. Explicit bases of both subspaces can be written by looking at the nine generators \eqref{eq:IO} of  the $E^9$ vector space that defines the $\Spin{9}$ structure. Namely, one has the compositions $$\J_{\alpha \beta} \ug \I_\alpha \I_\beta,$$ for $\alpha <\beta$ as a basis of $\Lambda^2_{36} \cong \liespin{9}$ and the compositions $\J_{\alpha \beta \gamma} \ug \I_\alpha \I_\beta \I_\gamma$ for $\alpha <\beta<\gamma$ as a basis of~$\Lambda^2_{84}$.

The K\"ahler $2$-forms ($\psi_{\alpha \beta}$) of the complex structures $\J_{\alpha \beta}$, obtained by denoting the coordinates in $\OO^2\cong\RR^{16}$ by $(1,\dots,8,1',\dots,8')$ are 
{
\small
\begin{myequation}\label{28}
\begin{aligned}
\psi_{12}&=(-\shortform{12}+\shortform{34}+\shortform{56}-\shortform{78})-(\enspace)\shortform{'}\enspace,\quad  &\psi_{13}&=(-\shortform{13}-\shortform{24}+\shortform{57}+\shortform{68})-(\enspace)\shortform{'}\enspace, \quad &\psi_{14}&=(-\shortform{14}+\shortform{23}+\shortform{58}-\shortform{67})-(\enspace)\shortform{'}\enspace, \\ 
\psi_{15}&=(-\shortform{15}-\shortform{26}-\shortform{37}-\shortform{48})-(\enspace)\shortform{'}\enspace, \quad &\psi_{16}&=(-\shortform{16}+\shortform{25}-\shortform{38}+\shortform{47})-(\enspace)\shortform{'}\enspace,\quad  &\psi_{17}&=(-\shortform{17}+\shortform{28}+\shortform{35}-\shortform{46})-(\enspace)\shortform{'}\enspace, \\
\psi_{18}&=(-\shortform{18}-\shortform{27}+\shortform{36}+\shortform{45})-(\enspace)\shortform{'}\enspace,\quad  &\psi_{23}&=(-\shortform{14}+\shortform{23}-\shortform{58}+\shortform{67})+(\enspace)\shortform{'}\enspace, \quad &\psi_{24}&=(\shortform{13}+\shortform{24}+\shortform{57}+\shortform{68})+(\enspace)\shortform{'}\enspace,\\ 
\psi_{25}&=(-\shortform{16}+\shortform{25}+\shortform{38}-\shortform{47})+(\enspace)\shortform{'}\enspace, \quad &\psi_{26}&=(\shortform{15}+\shortform{26}-\shortform{37}-\shortform{48})+(\enspace)\shortform{'}\enspace,\quad  &\psi_{27}&=(\shortform{18}+\shortform{27}+\shortform{36}+\shortform{45})+(\enspace)\shortform{'}\enspace, \\
\psi_{28}&=(-\shortform{17}+\shortform{28}-\shortform{35}+\shortform{46})+(\enspace)\shortform{'}\enspace,\quad  &\psi_{34}&=(-\shortform{12}+\shortform{34}-\shortform{56}+\shortform{78})+(\enspace)\shortform{'}\enspace, \quad &\psi_{35}&=(-\shortform{17}-\shortform{28}+\shortform{35}+\shortform{46})+(\enspace)\shortform{'}\enspace,\\ 
\psi_{36}&=(-\shortform{18}+\shortform{27}+\shortform{36}-\shortform{45})+(\enspace)\shortform{'}\enspace, \quad &\psi_{37}&=(+\shortform{15}-\shortform{26}+\shortform{37}-\shortform{48})+(\enspace)\shortform{'}\enspace,\quad  &\psi_{38}&=(\shortform{16}+\shortform{25}+\shortform{38}+\shortform{47})+(\enspace)\shortform{'}\enspace, \\
\psi_{45}&=(-\shortform{18}+\shortform{27}-\shortform{36}+\shortform{45})+(\enspace)\shortform{'}\enspace,\quad  &\psi_{46}&=(\shortform{17}+\shortform{28}+\shortform{35}+\shortform{46})+(\enspace)\shortform{'}\enspace, \quad &\psi_{47}&=(-\shortform{16}-\shortform{25}+\shortform{38}+\shortform{47})+(\enspace)\shortform{'}\enspace,\\ 
\psi_{48}&=(\shortform{15}-\shortform{26}-\shortform{37}+\shortform{48})+(\enspace)\shortform{'}\enspace, \quad &\psi_{56}&=(-\shortform{12}-\shortform{34}+\shortform{56}+\shortform{78})+(\enspace)\shortform{'}\enspace,\quad  &\psi_{57}&=(-\shortform{13}+\shortform{24}+\shortform{57}-\shortform{68})+(\enspace)\shortform{'}\enspace, \\
\psi_{58}&=(-\shortform{14}-\shortform{23}+\shortform{58}+\shortform{67})+(\enspace)\shortform{'}\enspace,\quad  &\psi_{67}&=(\shortform{14}+\shortform{23}+\shortform{58}+\shortform{67})+(\enspace)\shortform{'}\enspace,\quad &\psi_{68}&=(-\shortform{13}+\shortform{24}-\shortform{57}+\shortform{68})+(\enspace)\shortform{'}\enspace,\\ 
\psi_{78}&=(\shortform{12}+\shortform{34}+\shortform{56}+\shortform{78})+(\enspace)\shortform{'}\enspace,
\end{aligned}
\end{myequation}
}
where $(\enspace)\shortform{'}$ denotes the $\shortform{'}$ of what appears before it, for instance
{\small
\[
\psi_{12}=(-\shortform{12}+\shortform{34}+\shortform{56}-\shortform{78})-(-\shortform{1'2'}+\shortform{3'4'}+\shortform{5'6'}-\shortform{7'8'})\enspace.
\]
}
\indent Next,
{\small
\begin{equation}\label{8}
\begin{aligned}
\psi_{19}&=-\shortform{11'}-\shortform{22'}-\shortform{33'}-\shortform{44'}-\shortform{55'}-\shortform{66'}-\shortform{77'}-\shortform{88'}\enspace, 
&
\psi_{29}&=-\shortform{12'}+\shortform{21'}+\shortform{34'}-\shortform{43'}+\shortform{56'}-\shortform{65'}-\shortform{78'}+\shortform{87'}\enspace,
\\
\psi_{39}&=-\shortform{13'}-\shortform{24'}+\shortform{31'}+\shortform{42'}+\shortform{57'}+\shortform{68'}-\shortform{75'}-\shortform{86'}\enspace, 
&
\psi_{49}&=-\shortform{14'}+\shortform{23'}-\shortform{32'}+\shortform{41'}+\shortform{58'}-\shortform{67'}+\shortform{76'}-\shortform{85'}\enspace,
\\
\psi_{59}&=-\shortform{15'}-\shortform{26'}-\shortform{37'}-\shortform{48'}+\shortform{51'}+\shortform{62'}+\shortform{73'}+\shortform{84'}\enspace, 
&
\psi_{69}&=-\shortform{16'}+\shortform{25'}-\shortform{38'}+\shortform{47'}-\shortform{52'}+\shortform{61'}-\shortform{74'}+\shortform{83'}\enspace,
\\
\psi_{79}&=-\shortform{17'}+\shortform{28'}+\shortform{35'}-\shortform{46'}-\shortform{53'}+\shortform{64'}+\shortform{71'}-\shortform{82'}\enspace, 
&
\psi_{89}&=-\shortform{18'}-\shortform{27'}+\shortform{36'}+\shortform{45'}-\shortform{54'}-\shortform{63'}+\shortform{72'}+\shortform{81'}\enspace,
\end{aligned}.
\end{equation}
}
and a computation gives the following proposition. 
\begin{Proposition}\label{pr:charpoly}
The characteristic polynomial of the matrix $\psi=(\psi_{\alpha \beta})_{\alpha,\beta=1,\dots,9}$ of the K\"ahler forms explicitly listed in ~\eqref{28} and ~\eqref{8}, reduces to
\[
\det(tI-\psi)=t^9+\tau_4(\psi)t^5+\tau_8(\psi)t\enspace.
\]
\end{Proposition}

In particular, $\tau_2(\psi) = \sum_{\alpha < \beta} \psi_{\alpha \beta}^2 =0$, and the $\Spin{9}$-invariant $8$-form $\tau_4(\psi)$ has to be proportional to $\spinform{9}$. The proportionality factor, computed by looking at any of the terms of $\spinform{9}$ and $\tau_4(\psi)$ turns out to be $360$. 

This can be rephrased in the context of $\Spin{9}$ structures on Riemannian manifolds $M^{16}$ and gives the following two (essentially equivalent) algebraic expressions for the the 8-form, $\spinform{9}$:

\begin{Theorem}[\cite{LGMCEF}]
The $8$-form, $\spinform{9}$, associated with the $\Spin{9}$-structure $E^9\rightarrow M^{16}$ and defined by the integral \eqref{8form} coincides, up to a constant, with the global form
\[\label{cgm}
\Omega_{CGM}=\sum_{\alpha, \beta, \alpha', \beta'  = 1, \dots , 9} \psi_{\alpha , \beta} \wedge \psi_{\alpha , \beta'} \wedge \psi_{\alpha' , \beta} \wedge \psi_{\alpha' , \beta'} \enspace.
\] 
\end{Theorem}

\begin{Theorem}[\cite{PaPSAC}]
The $8$-form, $\spinform{9}$, associated with the $\Spin{9}$-structure $E^9\rightarrow M^{16}$ coincides, up to a~constant, with the coefficient 
\[ \label{pp}
\tau_4(\psi) =\sum_{1\leq \alpha_1 < \alpha_2 <  \alpha_3 <  \alpha_4 \leq 9} ( \psi_{\alpha_1 \alpha_2} \wedge \psi_{\alpha_3 \alpha_4} - \psi_{\alpha_1 \alpha_3} \wedge \psi_{\alpha_2  \alpha_4} + \psi_{\alpha_1 \alpha_4} \wedge \psi_{\alpha_2  \alpha_3} )^2\enspace
\] 
in the characteristic polynomial
\[
\det(tI-\psi)=t^9+\tau_4(\psi)t^5+\tau_8(\psi)t\enspace,
\]
where $\psi=(\psi_{\alpha \beta})_{\alpha,\beta=1,\dots,9}$ is any skew-symmetric matrix of local associated K\"ahler 2-forms ($M$). The~proportionality factor is given by
\[
360\spinform{9}=\tau_4(\psi)\enspace.
\]
\end{Theorem}

These two expressions of $\spinform{9}$ have been shown to be proportional according to the following algebraic relation.

\begin{Proposition}[\cite{LGMEEC}]\label{teo:main}
Let $\RR[x_{12}, \dots , x_{89}]$ be the polynomial ring in the 36 variables ($x_{12}, \dots , x_{89}$), and let $x$ be the skew-symmetric matrix whose upper diagonal entries are $x_{12}, \dots , x_{89}$. Among the homogeneous polynomials
\begin{equation*}
\begin{split} 
F=\sum_{\alpha, \beta, \alpha', \beta'  = 1, \dots , 9} x_{\alpha , \beta}  \; x_{\alpha , \beta'} \; x_{\alpha' , \beta}\;  x_{\alpha' , \beta'} \enspace , \qquad  \qquad P=\sum_{\alpha < \beta} x_{\alpha , \beta}^2   \enspace , \qquad  \\
Q=\tau_4(x) = \sum_{1\leq \alpha_1 < \alpha_2 <  \alpha_3 <  \alpha_4 \leq 9} ( x_{\alpha_1 \alpha_2} x_{\alpha_3 \alpha_4} - x_{\alpha_1 \alpha_3} x_{\alpha_2  \alpha_4} + x_{\alpha_1 \alpha_4} x_{\alpha_2  \alpha_3} )^2,
\end{split}
\end{equation*}
the following relation holds: $F=2P^2-4Q$. Thus, since $P(\psi)=0$,
$$\Omega_{CGM} =-4 \tau_4(\psi).$$
\end{Proposition}

\begin{Corollary}
The K\"ahler forms of the $\Spin{9}$-structure of $\OO^2$ allow the integral \eqref{8form} to be computed as
\[
\int_{\OP{1}}p_l^*\nu_l\,dl=\frac{\pi^4}{110880\cdot 360}\tau_4(\psi).
\]
\end{Corollary}

When $\Spin{9}$ is the holonomy group of the Riemannian manifold ($M^{16}$), the Levi--Civita connection ($\nabla$) preserves the $E^9$ vector bundle, and the local sections $\I_1,\dots,\I_9$ of $E^9$ induce the K\"ahler forms $\psi_{\alpha\beta}$ on $M$ as the local curvature forms.

\begin{Corollary}\label{hol} 
Let $M^{16}$ be a compact Riemannian manifold with holonomy $\Spin{9}$, i.e.,\ $M^{16}$ is either isometric to the Cayley projective plane ($\OP{2}$) or to any compact quotient of the Cayley hyperbolic plane ($\OH{2}$). Then, its~Pontrjagin classes are given by
\[
p_1(M)=0,\quad p_2(M)=-\frac{45}{2\pi^4}[\spinform{9}],\quad p_3(M)=0,\quad p_4(M)=-\frac{13}{256\pi^8}[\tau_8(\psi)].
\]
\end{Corollary}

\begin{proof}
The Pontrjagin classes of the $E^9\rightarrow M$ vector bundle are given by
\[
p_1(E)=0,\quad 16\pi^4 p_2(E)=\tau_4(\psi)=360[\spinform{9}],\quad p_3(E)=0,\quad 256\pi^8 p_4(E)=[\tau_8(\psi)].
\]

For a compact $M$ with a $\Spin{9}$-structure, the following relations hold ~(\cite{FriWSS} (p. 138)):
\begin{equation}
\begin{split}
p_1(M)&=2p_1(E), \\
p_2(M)&=\frac{7}{4}p_1^2(E)-p_2(E),\\
p_3(M)&=\frac{1}{8}\left(7p_1^3(E)-12p_1(E)p_2(E)+16p_3(E)\right),\\
p_4(M)&=\frac{1}{128}\left(35p_1^4(E)-120p_1^2(E)p_2(E)+400p_1(E)p_3(E)-1664p_4(E)\right).
\end{split}
\end{equation}

Thus, $\tau_2(\psi)=\tau_6(\psi)=0$ gives $p_1(E)=p_3(E)=0$, so that $p_1(M)=p_3(M)=0$, \mbox{$p_2(M)=-p_2(E)$}, and $p_4(M)=-13p_4(E)$.
\end{proof}

The Pontrjagin classes of $\OP{2}$ are known to be $p_2(\OP{2})=6u$ and $p_4(\OP{2})=39u^2$, where $u$ is the canonical generator of $H^8(\OP{2};\ZZ)$, and Corollary~\ref{hol} gives the representative forms:
\[
u=[-\frac{15}{4\pi^4}\spinform{9}]=[-\frac{1}{96\pi^4}\tau_4(\psi)],\qquad u^2=[-\frac{1}{768\pi^8}\tau_8(\psi)].
\]

\begin{Remark} Very recently, an alternative way of writing the 8-form $\spinform{9}$ in $\RR^{16}$ was proposed by J.~Kotrbat\'y~\cite{KotOVF}. This is in terms of the differentials of the octonionic coordinates $x,y \in \OO^2$. If 
\begin{equation*}
\begin{array}{rll}
& dx= dx_1 +idx_2+ jdx_3 + \dots +hdx_8, \quad &\overline{dx}= dx_1 -idx_2- jdx_3 - \dots -hdx_8,\\
& dy= dy_1 +idy_2+ jdy_3 + \dots +hdy_8, \quad &\overline{dy} = dy_1 -idy_2- jdy_3 , \dots -hdy_8,
\end{array}
\end{equation*}
consider, formally, the ``octonionic 4-forms''
\begin{equation*}
\begin{array}{rll}
&\Psi_{40}= ((\overline{dx} \wedge dx)\wedge \overline{dx})\wedge dx, \quad &\Psi_{31}= ((\overline{dy} \wedge dx)\wedge \overline{dx})\wedge dx \\
&\Psi_{13}= ((\overline{dx} \wedge dy)\wedge \overline{dy})\wedge dy, \quad  &\Psi_{04}= ((\overline{dy} \wedge dy)\wedge \overline{dy})\wedge dy. 
\end{array}
\end{equation*}

Then, by defining their conjugates through $$\overline{\alpha \wedge \beta}= (-1)^{kl} \; \overline{\beta} \wedge \overline{\alpha},$$ for $\alpha \in \Lambda^k, \beta \in \Lambda^l,$
Kotrbat\'y shows that the real 8-form 
\[
\Psi_8 = \Psi_{40} \wedge \overline{\Psi_{40}} +4 \Psi_{31} \wedge \overline{\Psi_{31}} -5 (\Psi_{31} \wedge \Psi_{13} + \overline{\Psi_{13}} \wedge \overline{\Psi_{31}}) +4 \Psi_{13} \wedge \overline{\Psi_{13}} +\Psi_{04} \wedge \overline{\Psi_{04}}
\]
gives the proportionality relation $\spinform{9}=-\frac{1}{4\cdot 6!} \Psi_8$ and recovers the table of $702$ non-zero monomials of $\spinform{9}$ in $\RR^{16}$ from this.
\end{Remark}

\section{The analogy with $\Sp{2} \cdot \Sp{1}$}\label{quaternionic}

In the previous Section we saw that the matrices $\I_1, \dots , \I_9$ are the starting point for the construction of the canonical 8-form $\spinform{9}$. Of course, $\I_1, \dots , \I_9$ are the octonionic analogues of the classical Pauli matrices
\begin{equation}\label{eq:Ipauli}
\I_1=\left(
\begin{array}{rr}
0 & 1 \\
1 & 0
\end{array}\right),\qquad
\I_2=\left(
\begin{array}{rr}
0 & -i \\
i & 0
\end{array}\right),\qquad
\I_3=\left(
\begin{array}{rr}
1 & 0 \\
0 & -1
\end{array}\right),
\end{equation}
which are defined with just the unit imaginary $i \in \CC$, belonging to $\U{2}$. Their compositions ($\J_{\alpha\beta}~=~\I_\alpha\I_\beta$), for $\alpha<\beta$ act
on $\HH\cong\CC^2$ as multiplications on the right by unit quaternions: $\J_{12}=R_i, \J_{13}= R_j, \J_{23}=R_k$. 

Similarly, the quaternionic analogues of the Pauli matrices are the $8 \times 8$ real matrices:

\begin{equation}\label{eq:IH}
\begin{aligned}
\I_1=&\left(
\begin{array}{c|c}
0 & \Id \\ \hline
\Id & 0
\end{array}
\right),\qquad
\I_2=\left(
\begin{array}{c|c}
0 & -R_i \\ \hline
R_i & 0
\end{array}
\right),\qquad 
\I_3=\left(
\begin{array}{c|c}
0 & -R_j \\ \hline
R_j & 0
\end{array}
\right),
\\
&\I_4=\left(
\begin{array}{c|c}
0 & -R_k \\ \hline
R_k & 0
\end{array}
\right), \qquad 
\I_5=\left(
\begin{array}{c|c}
\Id & 0 \\ \hline
0 & -\Id
\end{array}
\right),
\end{aligned}
\end{equation}
where $
R_i, 
R_j, and
R_k$  are the multiplication on the right on $\HH$ by $i,j,k$.

The ten compositions $\J_{\alpha\beta}\ug\I_\alpha\I_\beta$ $(\alpha<\beta)$ of these latter matrices are a basis of the term $\liesp{2}$ in the decomposition
\[
\Lambda^2\RR^8\cong\lieso{8}=\liesp{1}\oplus\liesp{2}\oplus\Lambda^2_{15},
\]
where $\liesp{2}\cong\lieso{5}$. Their K\"ahler forms $\theta_{\alpha\beta}$ read
{\small
\begin{equation*}
\begin{aligned}
\theta_{12} &= -\shortform{12}+\shortform{34}+\shortform{56}-\shortform{78}, &
\theta_{13} &= -\shortform{13}-\shortform{24}+\shortform{57}+\shortform{68}, &
\theta_{14} &= -\shortform{14}+\shortform{23}+\shortform{58}-\shortform{67},\\
\theta_{23} &= -\shortform{14}+\shortform{23}-\shortform{58}+\shortform{67}, &
\theta_{24} &= \shortform{13}+\shortform{24}+\shortform{57}+\shortform{68}, &
\theta_{34} &= -\shortform{12}+\shortform{34}-\shortform{56}+\shortform{78},\\
\theta_{15} &= -\shortform{15}-\shortform{26}-\shortform{37}-\shortform{48}, &
\theta_{25} &= -\shortform{16}+\shortform{25}+\shortform{38}-\shortform{47}, &
\theta_{35} &= -\shortform{17}-\shortform{28}+\shortform{35}+\shortform{46}, &
\theta_{45} &= -\shortform{18}+\shortform{27}-\shortform{36}+\shortform{45}.
\end{aligned}
\end{equation*}
}

If $\theta\ug(\theta_{\alpha\beta})$, it follows that
\begin{equation}\label{Theta}
\tau_2(\theta) =\sum_{\alpha < \beta} \theta^2_{\alpha \beta} = -12\shortform{1234}-4\shortform{1256}-4\shortform{1357}+4\shortform{1368}-4\shortform{1278}-4\shortform{1467}-4\shortform{1458}+\star =-2 \Omega_L 
\end{equation}
where $\star$ denotes the Hodge star of what appears before, and $$\Omega_L = \omega^2_{L_i}+\omega^2_{L_j}+\omega^2_{L_k}$$ is the left quaternionic 4-form on $\HH^2 = \RR^8$.

On the other hand, the matrices 
$
B=\left(
\begin{array}{c|c}
B' & B'' \\ 
\hline 
B''' & B'''' 
\end{array}
\right)\in\SO{8}
$ which commute with each of the involutions $\I_1,\dots,\I_5$ are the ones satisfying $B''=B'''=0$ and $B'=B''''\in\Sp{1}\subset\SO{4}$. Thus, the~subgroup preserving each of the $\I_1,\dots,\I_5$ is the diagonal $\Sp{1}_\Delta \subset\SO{8}$. Thus, the subgroup of $\SO{8}$ preserving the vector bundle $E^5$ consists of matrices ($B$) satisfying $$B\I_\alpha=\I'_\alpha B,$$ with $\I_1,\dots,\I_5$ and $\I'_1,\dots,\I'_5$ bases of $E^5$ related by a $\SO{5}$ matrix. This group is thus recognized to be $\Sp{1}\cdot\Sp{2}$, and the following proposition applies. 

\begin{Proposition}
Let $M^8$ be an 8-dimensional oriented Riemannian manifold, and let $E^5$ be a vector subbundle of $ \End{TM}$, locally spanned by self dual anti-commuting involutions ($\I_1,\dots,\I_5$) and related, on open sets covering $M$, by functions giving $\SO{5}$ matrices. Then, the datum of such an $E^5$ is equivalent to an (left) almost quaternion Hermitian structure on $M$, i.e.,\ to a $\Sp{2}\cdot\Sp{1}$-structure on $M$.
\end{Proposition}

In the above discussions, we looked at the standard  $\U{2}$ and $\Sp{1}\cdot\Sp{2}$-structures on $\RR^4$ and $\RR^8$, through the decompositions of the $2$-forms
\[
\lieso{4}=\lieu{1}\oplus\lieso{3}\oplus\Lambda^2_2\enspace,\qquad\lieso{8}=\liesp{1}\oplus\liesp{2}\oplus\Lambda^2_{15},
\]
and the orthonormal frames in the $\lieso{3}$ and $\liesp{2}$ components, respectively. The last components, $\Lambda^2_2$~and $\Lambda^2_{15}$, describe all of the similar structures on the linear spaces $\RR^4$ and $\RR^8$. Thus, such~decompositions give rise to the $\SO{4}/\U{2}$ and $\SO{8}/\Sp{1}\cdot\Sp{2}$ spaces---the spaces of all possible structures in the two cases.

To summarize (cf, \cite{GWZGHF}),

\begin{Corollary}\label{involutions}
The actions $
\left(
\begin{array}{c} 
x \\ 
x'
\end{array}
\right)\longrightarrow
\left(
\begin{array}{cc}
r & R_{\overline u} \\
R_u & -r
\end{array}
\right)  
\left(
\begin{array}{c}
x \\
x'
\end{array}
\right)\enspace $, when $u,x,x' \in \CC,\HH,\OO$ (and, in any case, $r \in \RR$ and $r^2+u\overline u =1$) generate the groups $\U{2}$, $\Sp{2}\cdot\Sp{1}$, $\Spin{9}$ of symmetries
of the Hopf fibrations
\[
S^3 \longrightarrow S^2, \qquad S^7 \longrightarrow S^4, \qquad S^{15} \longrightarrow S^8.
\]

The corresponding $G$-structures on the Riemannian manifolds $M^4$, $M^8$, $M^{16}$ can be described through $E \subset \mathrm{End} \; TM$ vector subbundles of ranks $3,5,9$, respectively. Any such $E$ is locally generated by the self-dual involutions $\mathcal I_\alpha$ satisfying $\mathcal I_\alpha \mathcal I_\beta = -\mathcal I_\beta \mathcal I_\alpha$ for $\alpha \neq \beta$ and related, on open neighborhoods covering $M$, by~functions that give matrices in $\SO{3}$, $\SO{5}$, and $\SO{9}$.
\end{Corollary}

%
%
%
%
%

\section{Vector fields on spheres}\label{Spheres}

An application of $\Spin{9}$ structures is the possibility of writing a maximal orthonormal system of tangent vector fields on spheres of any dimension. Here, we outline this construction only on some ``low-dimensional'' cases (in fact. up to the $S^{511}$ sphere), referring, for the general case, to the linear algebra formalism developed in Ref. \cite{PaPSMS}.

Recall that the identifications $\RR^{2n} = \CC^n$,  $\RR^{4n} = \HH^n$ and $\RR^{8n} = \OO^n$ allow to act on the normal vector field of the unit sphere by the imaginary units of $\CC,\HH,\OO$, giving $1$, $3$, and $7$ tangent orthonormal vector fields on $S^{2n-1}$, $S^{4n-1}$, and $S^{8n-1}$. These are, in fact, a maximal system of linearly independent vector fields on $S^{m-1} \subset \RR^m$, provided the (even) dimension ($m$) of the ambient space is not divisible by $16$. The maximal number ($\sigma(m)$) of linearly independent vector fields on any $S^{m-1}$ is well-known to be expressed as
\[
\sigma(m) = 2^p + 8q -1\,
\]
where $\sigma(m) + 1 = 2^p + 8q$ is the \emph{Hurwitz--Radon number}, referring to the decomposition
\begin{equation}\label{eq:dec}
m = (2k+1)2^p 16^q, \qquad \text{where } 0 \leq p \leq 3,
\end{equation}
(cf. Ref. \cite{PaPSMS} for further information and references on this classical subject).

Table~\ref{somespheres} lists some of the lowest dimensional $S^{m-1} \subset \RR^m$ spheres that admit a maximal number $\sigma(m)>7$ of linearly independent vector fields.

\begin{table}[H]
\caption{Some $S^{m-1}$ spheres with more than seven vector fields.}\label{somespheres}
\begin{center}
{
\begin{tabular}{cccccccccccccccccccccc}
\toprule
$m-1$ & $15$ & $31$ & $47$ & $63$ & $79$ & $95$ & $111$ & $127$ & $143$&$159$&$175$&$191$& \dots & $255$ & \dots & $511$ & \dots\\
\midrule 
$\sigma(m)$ & $8$ & $9$ & $8$ & $11$ & $8$ & $9$ & $8$ & $15$ & $8$&$9$&$8$&$11$& \dots & $16$ & \dots & $17$ & \dots\\
\bottomrule
\end{tabular}
}
\end{center}
\end{table}

The first of them is $S^{15} \subset \RR^{16}$, which is acted on by $\Spin{9}\subset \SO{16}$. To write the eight vector fields on $S^{15}$, it is convenient to look at the involutions $\I_1, \dots ,\I_9$ and at the eight complex structures ($\J_{1},\dots,\J_{8}$) on $\RR^{16}$:
\[
\J_{\alpha}\ug\I_{\alpha}\I_{9}: \RR^{16} \longrightarrow \RR^{16}\enspace, \qquad \alpha =1,\dots,8.
\]

Denote by
\[
N\ug (x,y)\ug(x_1,\dots,x_8,y_1,\dots,y_8) 
\]
the (outward) unit normal vector field of $S^{15} \subset \RR^{16}$. Then, the following proposition applies.

\begin{Proposition}
The vector fields
{
\begin{equation}\label{otto}
\begin{split}
\J_{1}N =&(-y_1,-y_2,-y_3,-y_4,-y_5,-y_6,-y_7,-y_8,x_1,x_2,x_3,x_4,x_5,x_6,x_7,x_8),\\
\J_{2}N =&(-y_2,y_1,y_4,-y_3,y_6,-y_5,-y_8,y_7,-x_2,x_1,x_4,-x_3,x_6,-x_5,-x_8,x_7),\\
\J_{3}N =&(-y_3,-y_4,y_1,y_2,y_7,y_8,-y_5,-y_6,-x_3,-x_4,x_1,x_2,x_7,x_8,-x_5,-x_6),\\
\J_{4}N =&(-y_4,y_3,-y_2,y_1,y_8,-y_7,y_6,-y_5,-x_4,x_3,-x_2,x_1,x_8,-x_7,x_6,-x_5),\\
\J_{5}N =&(-y_5,-y_6,-y_7,-y_8,y_1,y_2,y_3,y_4,-x_5,-x_6,-x_7,-x_8,x_1,x_2,x_3,x_4),\\
\J_{6}N =&(-y_6,y_5,-y_8,y_7,-y_2,y_1,-y_4,y_3,-x_6,x_5,-x_8,x_7,-x_2,x_1,-x_4,x_3),\\
\J_{7}N =&(-y_7,y_8,y_5,-y_6,-y_3,y_4,y_1,-y_2,-x_7,x_8,x_5,-x_6,-x_3,x_4,x_1,-x_2),\\
\J_{8}N =&(-y_8,-y_7,y_6,y_5,-y_4,-y_3,y_2,y_1,-x_8,-x_7,x_6,x_5,-x_4,-x_3,x_2,x_1)
\end{split}
\end{equation}
}
are tangent to $S^{15}$ and orthonormal.
\end{Proposition}

Indeed, by fixing any $\beta$, $1\leq \beta \leq 9$ and considering the $8$ complex structures $\I_\alpha\I_\beta$ with $\alpha \neq \beta$, the~eight vector fields
($\I_\alpha\I_\beta N$) are still tangential to $S^{15}$ and orthonormal.

Although it is well-known that $\CC, \HH$, and $\OO$ are the only normed algebras over $\RR$,
to move to a~higher dimension, it is convenient to consider the algebra $\mathbb{S}$ of sedenions, obtained from $\OO$ through the Cayley--Dickson process. 
By denoting the canonical basis of $\Se$ over $\RR$ by $1, e_1,\dots ,e_{15}$, one can write a multiplication table (cf. Ref. \cite{PaPSMS}). An example of divisors of the zero in $\Se$ is given by $(e_2 - e_{11})(e_7 + e_{14})=0$.

The following remark helps in higher dimensions. Consider the $S^{m-1} \subset\RR^m$ sphere, and~decompose $m$ as $m=(2k+1)2^p16^q$, where $p\in\{0,1,2,3\}$. First, observe that a vector field ($B$) that is tangential to the $S^{2^p16^q-1}\subset\RR^{2^p16^q}$ sphere induces a vector field
\begin{equation}\label{eq:blockwise}
\underbrace{(B,\dots,B)}_{2k+1\text{ times}}
\end{equation}
that is tangential to the $S^{(2k+1)2^p16^q-1}$ sphere. Thus, assume in what follows that $k=0$, i.e., $m=2^p16^q$. Whenever we extend a vector field in this way, we call the vector field given by~\eqref{eq:blockwise} the \emph{diagonal extension} of $B$.

If $q=0$, that is, if $m$ is not divisible by $16$, the vector fields on $S^{m-1}$ are given by the complex, quaternionic, or octonionic multiplication for $p=1, 2$, or $3$ respectively, so that the $\Spin{9}$ contribution occurs when $q \geq 1$, that is, $m=16l$, and we can denote the coordinates in $\RR^{16l}$ by $(s^1,\dots,s^l)$, where~each $s^\alpha$, for $\alpha=1,\dots,l$, belongs to the sedenions ($\Se$), and can thus be identified with a pair ($(x^\alpha,y^\alpha)$) of octonions.

The unit (outward) normal vector field ($N$) of $S^{16l-1}$ can be denoted by using the sedenions:
\[
N \ug (s^1,\dots,s^l)\quad\text{where}\quad\norm{s^1}^2+\dots+\norm{s^l}^2=1.
\]

Therefore, we can think of $N$ as an element of $\Se^l=\OO^{2l}=\RR^{16l}$.

Whenever $l=2, 4$, or $8$, denoted by $\coniugiobaseraw$, the following automorphism of $\Se^l=\OO^{2l}$ applies.
\begin{equation}\label{star}
\coniugiobaseraw: ((x^1,y^1),\dots,(x^l,y^l))\longrightarrow ((x^1,-y^1),\dots, (x^l,-y^l)).
\end{equation}

We refer to $\coniugiobaseraw$ as a \emph{conjugation}, due to its similarity to that in $*$-algebras.

Moreover, it is convenient to use the following formal notations:
\begin{align}
N&=(s^1,s^2)\ug s^1+i s^2 \in S^{31},\label{eq:formalC}\\
N&=(s^1,s^2,s^3,s^4)\ug s^1+i s^2 +j s^3 +k s^4 \in S^{63},\label{eq:formalH}\\
N&=(s^1,s^2,s^3,s^4,s^5,s^6,s^7,s^8)\ug s^1+i s^2 +j s^3 +k s^4+e s^5+f s^6+g s^7+h s^8 \in S^{127},\label{eq:formalO}
\end{align}
which allow us to define left multiplications ($\LF$) in the sedenionic spaces $\Se^{2}$, $\Se^{4}$, and $\Se^{8}$ (like in $\CC$, $\HH$, and $\OO$), as follows.

If $l=2$, the left multiplication is
\begin{equation}\label{i}
\LF_i(s^1,s^2)\ug -s^2+is^1,
\end{equation}
whereas, if $l=4$, we define
\begin{equation}\label{ijk}
\begin{split}
\LF_i(s^1,\dots,s^4)&\ug -s^2+i s^1-js^4+ k s^3,\\
\LF_j(s^1,\dots,s^4)&\ug -s^3+i s^4+js^1- k s^2,\\
\LF_k(s^1,\dots,s^4)&\ug -s^4-i s^3+js^2+ k s^1,
\end{split}
\end{equation}
and finally, if $l=8$, we define
\begin{equation}\label{ijkefgh}
\begin{split}
\LF_i(s^1,\dots,s^8)&\ug -s^2+i s^1-js^4+ k s^3- e s^6+f s^5+gs^8-h s^7, \\
\LF_j(s^1,\dots,s^8)&\ug -s^3+i s^4+js^1- k s^2 - e s^7-f s^8 +g s^5+ h s^6, \\
\LF_k(s^1,\dots,s^8)&\ug -s^4-i s^3+js^2+ k s^1- e s^8+f s^7- g s^6+ h s^5, \\
\LF_e(s^1,\dots,s^8)&\ug -s^5+i s^6+js^7+ k s^6+e s^1-f s^2-gs^3 -h s^4, \\
\LF_f(s^1,\dots,s^8)&\ug -s^6-i s^5+js^8- k s^7+es^2+f s^1+gs^4-h s^3, \\
\LF_g(s^1,\dots,s^8)&\ug -s^7-i s^8-js^5+ k s^6+es^3-f s^4+gs^1+ h s^2, \\
\LF_h(s^1,\dots,s^8)&\ug -s^8+i s^7-js^6- k s^5 +es^4+f s^3-gs^2+ h s^1.
\end{split}
\end{equation}

Note that, in all three cases, $l=2, 4$, and $8$, and the vector fields $\LF_i(N),\dots,\LF_h(N)$ are tangential to $S^{31}$, $S^{63}$, and $S^{127}$, respectively.

We can now write the maximal systems of vector fields on $S^{31}$, $S^{63}$ and $S^{127}$ as follows.

\vspace{6pt}
\noindent{\bf Case $\mathbf{p=1}$}
\vspace{6pt}

For $S^{31}$, whose maximal number of tangent vector fields is nine, we obtain eight vector fields by writing the unit normal vector field as $N=(s^1,s^2)=(x^1,y^1,x^2,y^2)\in S^{31}\subset \Se^{2}$, where $x^1,y^1,x^2,y^2\in\OO$, and repeating Formula~\eqref{otto} for each pair ($(x^1,y^1),(x^2,y^2)$):
{
\begin{myequation2}\label{otto'}
\begin{split}
\J_{1}N=(\J_{1}s^1,\J_{1}s^2) &=(-y^1_1,-y^1_2,\dots ,-y^1_7 ,-y^1_8,x^1_1,x^1_2,\dots ,x^1_7 ,x^1_8,-y^2_1,-y^2_2, \dots ,-y^2_7,-y^2_8,x^2_1,x^2_2,\dots ,x^2_7 ,x^2_8),\\
\J_{2}N=(\J_{2}s^1,\J_{2}s^2) &=(-y^1_2,y^1_1,\dots ,-y^1_8,y^1_7,-x^1_2,x^1_1,\dots -x^1_8,x^1_7,-y^2_2,y^2_1,\dots ,-y^2_8 ,y^2_7,-x^2_2,x^2_1,\dots ,-x^2_8 ,x^2_7),\\
\J_{3}N=(\J_{3}s^1,\J_{3}s^2) &=(-y^1_3,-y^1_4,\dots ,-y^1_5,-y^1_6,-x^1_3,-x^1_4,\dots ,-x^1_5,-x^1_6,-y^2_3,-y^2_4,\dots ,-y^2_5,-y^2_6,-x^2_3,-x^2_4,\dots ,-x^2_5,-x^2_6),\\
\J_{4}N=(\J_{4}s^1,\J_{4}s^2) &=(-y^1_4,y^1_3,\dots ,y^1_6 ,-y^1_5,-x^1_4,x^1_3,\dots ,x^1_6 ,-x^1_5,-y^2_4,y^2_3,\dots ,y^2_6,-y^2_5,-x^2_4,x^2_3,\dots ,x^2_6,-x^2_5),\\
\J_{5}N=(\J_{5}s^1,\J_{5}s^2) &=(-y^1_5,-y^1_6,\dots ,y^1_3 ,y^1_4,-x^1_5,-x^1_6,\dots ,x^1_3,x^1_4,-y^2_5,-y^2_6,\dots ,y^2_3,y^2_4,-x^2_5,-x^2_6,\dots ,x^2_3 ,x^2_4),\\
\J_{6}N=(\J_{6}s^1,\J_{6}s^2) &=(-y^1_6,y^1_5,\dots ,-y^1_4 ,y^1_3,-x^1_6,x^1_5,\dots ,-x^1_4 ,x^1_3,-y^2_6,y^2_5,\dots ,-y^2_4 ,y^2_3,-x^2_6,x^2_5,\dots ,-x^2_4 ,x^2_3),\\
\J_{7}N=(\J_{7}s^1,\J_{7}s^2) &=(-y^1_7,y^1_8,\dots ,y^1_1 ,-y^1_2,-x^1_7,x^1_8,\dots ,x^1_1 ,-x^1_2,-y^2_7,y^2_8,\dots ,y^2_1 ,-y^2_2,-x^2_7,x^2_8,\dots ,x^2_1,-x^2_2),\\
\J_{8}N=(\J_{8}s^1,\J_{8}s^2) &=(-y^1_8,-y^1_7,\dots ,y^1_2 ,y^1_1,-x^1_8,-x^1_7, \dots ,x^1_2 ,x^1_1,-y^2_8,-y^2_7,\dots ,y^2_2 ,y^2_1,-x^2_8,-x^2_7, \dots ,x^2_2 ,x^2_1). 
\end{split}
\end{myequation2}
}

A ninth orthonormal vector field, completing the maximal system, is found by the formal left multiplication~\eqref{i} 
and the $\coniugiobaseraw$ automorphism~\eqref{star}:
\begin{equation}
\coniugiobaseraw(\LF_i N)=\coniugiobaseraw(-s^2,s^1)=(-x^2,y^2,x^1,-y^1).
\end{equation}

\vspace{6pt}
\noindent{\bf Case $\mathbf{p=2}$}
\vspace{6pt}

The $S^{63}$ sphere has a maximal number of $11$ orthonormal vector fields. The normal vector field is, in this case, given by $N=(s^1,\dots,s^4)=(x^1,y^1,\dots,x^4,y^4)\in S^{63} \subset \Se^{4}$, and eight vector fields arise as $\J_{\alpha}N$, for $\alpha=1,\dots,8$. Three other vector fields are again given by the formal left multiplications~\eqref{ijk} and the $\coniugiobaseraw$ automorphism in~\eqref{star}:
\begin{equation}
\begin{split}
\coniugiobaseraw(\LF_i N) &=(-x^2,y^2,x^1,-y^1,-x^4,y^4,x^3,-y^3),\\ 
\coniugiobaseraw (\LF_j N) &=(-x^3,y^3,x^4,-y^4,x^1,-y^1,-x^2,y^2),\\
\coniugiobaseraw (\LF_k N) &=(-x^4,y^4,-x^3,y^3,x^2,-y^2,x^1,-y^1).
\end{split}
\end{equation}

\vspace{6pt}
\noindent{\bf Case $\mathbf{p=3}$}
\vspace{6pt}

The $S^{127}$ sphere has a maximal number of $15$ orthonormal vector fields. Eight of them are still given by $\J_{\alpha}N$ for $\alpha=1,\dots,8$, whereas the formal left multiplications given in~\eqref{ijkefgh} yield the seven tangent vector fields ($\coniugiobaseraw(\LF_\alpha N)$), for $\alpha\in\{i,\dots,h\}$.

\vspace{6pt}
\noindent{\bf The $S^{255}$\label{sec:256} sphere} 
\vspace{6pt}

To write a system of $16$ orthonormal vector fields on $S^{255} \subset \RR^{256}$, decompose
\begin{equation}\label{256}
\RR^{256} = \RR^{16} \oplus \dots \oplus \RR^{16}
\end{equation}
into sixteen components. The unit outward normal vector field is
\[
N=(s^1,\dots,s^{16}),
\]
where $s^1,\dots, s^{16}$ are sedenions.

The matrices in $M_{16}(\RR)$ which give the complex structures $\J_{1},\dots,\J_{8}$ act on $N$ not only separately on each of the $16$-dimensional components of~\eqref{256}, but also formally on the (column) $16$-ples of sedenions $(s^1,\dots,s^{16})^T$. Based on which of the two actions of the same matrices are considered in $\RR^{256}$, we use the notations 
\[
\J_{1},\dots,\J_{8} \qquad \text{or} \qquad \blocktext(\J_{1}),\dots, \blocktext(\J_{8}),
\]
in both cases being all complex structures on $\RR^{256}$. The following $16$ vector fields are obtained: 
\begin{align}
\J_{1}N\enspace,\enspace&\dots\enspace,\J_{8}N\enspace,\label{eq:level1}\\
\coniugiobaseraw(\blocktext(\J_{1})N)\enspace,\enspace&\dots\enspace,\coniugiobaseraw(\blocktext(\J_{8})N)\enspace,\label{eq:level2}
\end{align}
where $\coniugiobaseraw$ is defined in Formula~\eqref{star}.
The \emph{level $1$ vector fields} and \emph{level $2$ vector fields} are the ones given by \eqref{eq:level1} and \eqref{eq:level2}, respectively. Then, the following proposition applies.

\begin{Proposition}\label{16}
Formulas~\eqref{eq:level1} and \eqref{eq:level2} give a maximal system of $16$ orthonormal tangent vector fields on $S^{255}$.
\end{Proposition}

\begin{proof}
Denote sedenions as pairs $s^\alpha\ug(x^\alpha,y^\alpha)$ of octonions. 
The unit normal vector field is
\begin{equation}
N=(s^1,\dots,s^{16})=(x^1,y^1,\dots,x^{16},y^{16}) \in S^{255},
\end{equation}
and one gets the following tangent vectors that can be easily checked to be orthonormal:
\begin{equation}\label{B}
\begin{split}
\J_{1}N&=(\J_{1} s^1,\dots,\J_{1}s^{16}) = (-y^1,x^1,\dots,-y^{16},x^{16}),\\
\J_{2}N&=(\J_{2} s^1,\dots,\J_{2}s^{16}) = (R_i y^1,R_i x^1,\dots,R_i y^{16}, R_i x^{16}),\\
\J_{3}N&=(\J_{3} s^1,\dots,\J_{3}s^{16}) = (R_j y^1,R_j x^1,\dots,R_j y^{16}, R_j x^{16}),\\
\J_{4}N&=(\J_{4} s^1,\dots,\J_{4}s^{16}) = (R_k y^1,R_k x^1,\dots,R_k y^{16}, R_k x^{16}),\\
\J_{5}N&=(\J_{5} s^1,\dots,\J_{5}s^{16}) = (R_e y^1,R_e x^1,\dots,R_e y^{16}, R_e x^{16}),\\
\J_{6}N&=(\J_{6} s^1,\dots,\J_{6}s^{16}) = (R_f y^1,R_f x^1,\dots,R_f y^{16}, R_f x^{16}),\\
\J_{7}N&=(\J_{7} s^1,\dots,\J_{7}s^{16}) = (R_g y^1,R_g x^1,\dots,R_g y^{16}, R_g x^{16}),\\
\J_{8}N&=(\J_{8} s^1,\dots,\J_{8}s^{16}) = (R_h y^1,R_h x^1,\dots, R_h y^{16}, R_h x^{16}).
\end{split}
\end{equation}

Moreover, one obtains eight further vector fields which can be similarly verified to be orthonormal.
\begin{equation}\label{B'}
\begin{split}
\coniugiobaseraw(\blocktext(\J_{1})N)&=\coniugiobaseraw(-s^9,-s^{10},-s^{11},-s^{12},-s^{13},-s^{14},-s^{15},-s^{16}, s^1,s^2,s^3,s^4,s^5,s^6,s^7,s^8),\\
\coniugiobaseraw(\blocktext(\J_{2})N)&=\coniugiobaseraw(-s^{10}, s^9,s^{12},-s^{11},s^{14},-s^{13}, -s^{16},s^{15}, -s^2,s^1,s^4,-s^3,s^6,-s^5,-s^8,s^7),\\
\coniugiobaseraw(\blocktext(\J_{3})N)&= \coniugiobaseraw(-s^{11},-s^{12},s^{9},s^{10},s^{15},s^{16}, -s^{13},-s^{14},- s^3,-s^4,s^1,s^2,s^7,s^8,-s^5,-s^6),\\
\coniugiobaseraw(\blocktext(\J_{4})N)&=\coniugiobaseraw(-s^{12},s^{11},-s^{10},s^{9},s^{16},-s^{15}, s^{14},-s^{13}, -s^4,s^3,-s^2,s^1,s^8,-s^7,s^6,-s^5),\\
\coniugiobaseraw(\blocktext(\J_{5})N)&= \coniugiobaseraw(-s^{13},-s^{14},-s^{15},-s^{16},s^{9},s^{10}, s^{11},s^{12},- s^5,-s^6,-s^7,-s^8,s^1,s^2,s^3,s^4),\\
\coniugiobaseraw(\blocktext(\J_{6})N)&= \coniugiobaseraw(-s^{14},s^{13},-s^{16},s^{15},-s^{10},s^{9}, -s^{12},s^{11}, -s^6,s^5,-s^8,s^7,-s^2,s^1,-s^4,s^3),\\
\coniugiobaseraw(\blocktext(\J_{7})N)&= \coniugiobaseraw(-s^{15},s^{16},s^{13},-s^{14},-s^{11},s^{12}, s^{9},-s^{10}, -s^7,s^8,s^5,-s^6,-s^3,s^4,s^1,-s^2),\\
\coniugiobaseraw(\blocktext(\J_{8})N)&= \coniugiobaseraw(-s^{16},-s^{15},s^{14},s^{13},-s^{12},-s^{11}, s^{10},s^{9}, -s^8,-s^7,s^6,s^5,-s^4,-s^3,s^2,s^1),\\
\end{split}
\end{equation}

To see that each $\J_{\alpha}N$ vector is orthogonal to each $\coniugiobaseraw(\blocktext(\J_{\beta})N)$, for $\alpha,\beta=1,\dots,8$, look at the matrix representations of $R_i,\dots,R_h$ and write the octonionic coordinates as $x^\lambda = h_1^\lambda + h_2^\lambda e$, $y^\mu~=~k_1^\mu~+~k_2^\mu e$. Then, the scalar product $<\J_{\alpha}N,\coniugiobaseraw(\blocktext(\J_{\beta})N)>$ can be computed with Formula~\eqref{oct} to obtain the product of the octonions. For example, recall from Formula~\eqref{B} that
\[
\J_{8}N=(y^1h,x^1h,\dots,y^8h,x^8h,y^9h,x^9h,\dots,y^{16}h,x^{16}h),
\]
so that the computation of $<\J_{8}N,\coniugiobaseraw(\blocktext(\J_{1})N)>$ gives rise to pairs of terms like in
\[
\begin{split}
<\J_{8}N,\coniugiobaseraw(\blocktext(\J_{1})N)> = \Re(-(R_hy^1)\overline{x}^9 -(R_h x^9) \overline{y}^1 + \dots) = \\ = \Re(\underline{-k k_2^1 \overline{h}_1^9} - \underline{\underline{\overline{h}_2^9 k k_1^1}}-\underline{\underline{k h_2^9\overline{k}_1^1}}- \underline{\overline{k}_2^1 k h_1^9}  + \dots).
\end{split}
\]

To conclude, observe that the real part ($\Re$) of the sums of each of the corresponding underlined terms is zero. This is due to the identity ($\Re(h h' h'') = \Re(h' h'' h)$), that holds for all $h,h',h'' \in \HH$.
\end{proof}

More generally, the following proposition applies.

\begin{Proposition}
Fix any $\beta$, $1\leq \beta \leq 9$, consider the eight complex structures $\I_{\alpha}\I_{\beta}$, with $\alpha \neq \beta$, defined on $\RR^{256} =\RR^{16} \oplus \dots \oplus \RR^{16}$ by acting with the corresponding matrices on the listed $16$-dimensional components, that is, the diagonal extension of $\I_{\alpha}\I_{\beta}$. Also, consider the further eight complex structures $\coniugiobaseraw(\blocktext(\I_{\alpha}\I_{\beta}))$ for $\alpha \neq \beta$, defined by the same matrices that now act on the column matrix of sedenions ($(s^1, \dots s^{16})^T$). Then,
\[
\{\I_{\alpha}\I_{\beta}N,\coniugiobaseraw(\blocktext(\I_{\alpha}\I_{\beta})N)\}_{\alpha\neq\beta}
\]
is a maximal system of $16$ orthonormal tangent vector fields on $S^{255}$.
\end{Proposition}

The $m=2\cdot 16^2$ dimension, that is, the $S^{511}$ sphere, is the lowest dimensional case where the last ingredient of our construction enters. To define the additional vector field here, we need to extend the formal left multiplication defined by Formula~\eqref{i}. Consider the decomposition
\[
\RR^{2\cdot 16^2}=\RR^{16^2}\oplus\RR^{16^2}
\]
and denote the elements in $\RR^{16^2}$ by $s^1,s^2$
. Using the notation
\begin{align}
N&=(s^1,s^2)\ug s^1+i s^2 \in S^{2\cdot 16^2-1}\enspace,\label{eq:formalCgeneral}
\end{align}
a formal left multiplication ($\LF_i$ in $\RR^{2\cdot 16^2}$) can be defined with Formula~\eqref{i}. It can then be expected that $\coniugiobaseraw(\LF_iN)$ is orthogonal to $\{\J_{\alpha}N,\coniugiobaseraw(\blocktext(\J_{\alpha})N)\}_{\alpha=1,\dots,8}$, but this is not the case. In fact, $\coniugiobaseraw(\LF_iN)$ appears to be orthogonal to the first eight vector fields but not to the second ones. 

To make things work, we need to extend not only $\LF_i$, but also the D conjugation. To this aim, the~elements $s^\alpha\in\RR^{16^2}$ are split into $(x^\alpha,y^\alpha)$ where $x^\alpha,y^\alpha\in\RR^{16^2/2}$, and the conjugation $\coniugiobase{2}$ is defined on $\RR^{16^2}$ using Formula~\eqref{star}:
\begin{equation}\label{eq:coniugiobase2}
\coniugiobase{2}:((x^1,y^1),(x^2,y^2))\longrightarrow((x^1,-y^1),(x^2,-y^2)).
\end{equation}

The additional vector field is then $\coniugiobaseraw(\coniugiobase{2}(\LF_iN))$, and the following theorem applies.
\begin{Theorem}
A maximal orthonormal system of tangent vector fields on $S^{2\cdot 16^2-1}$ is given by the following $8\cdot 2+1$ vector fields:
\begin{equation}
\begin{split}
\J_{1}N\enspace,\enspace&\dots\enspace,\J_{8}N\enspace,\\
\coniugiobaseraw(\blocktext(\J_{1})N)\enspace,\enspace&\dots\enspace,\coniugiobaseraw(\blocktext(\J_{8})N)\enspace,\\
\coniugiobaseraw(\coniugiobase{2}&(\LF_iN)).
\end{split}
\end{equation}
\end{Theorem}

\section{Back to the octonionic Hopf fibration}\label{Hopf}

As we saw in the previous Section, $S^{15}$ is the lowest dimensional sphere with more than seven linearly independent vector fields. There are further features that distinguish $S^{15}$ among spheres. For~example, $S^{15}$ is the only sphere that admits three homogeneous Einstein metrics, and it is the only sphere that appears as a regular orbit in three cohomogeneity actions on projective spaces, namely, of $\SU{8}$, $\Sp{4}$, and $\Spin{9}$ on $\CP{8}$, $\HP{4}$, and $\OP{2}$, respectively (see ~Refs. \cite{BesEiM,KolCHC}). All of these features can be traced back to the transitive action of $\Spin{9}$ on the octonionic Hopf fibration $S^{15}\rightarrow S^{8}$. The following theorem applies.

\begin{Theorem}\label{te:A}{}
Any global vector field on $S^{15}$ which is tangent to the fibers of the octonionic Hopf fibration $S^{15} \rightarrow S^{8}$ has at least one zero.
\end{Theorem}

\begin{proof}\hypertarget{pr:proofA}{}
For any $(x,y)\in S^{15}\subset\OO^2=\RR^{16}$, we already denoted by 
\[
N = (x,y) = (x_1,\dots,x_8,y_1,\dots,y_8) 
\]
the (outward) unit normal vector field of $S^{15}$ in $\RR^{16}$. After identifying the tangent spaces $T_{(x,y)}(\RR^{16})$ with $\RR^{16}$, it can be noted that
the $\I_1,\dots,\I_9$ involutions define the following 
sections of~$T(\RR^{16})_{\vert_{S^{15}}}$:
{ \begin{myequation2}\label{spin9B}
\begin{split}
\I_1 N &=(y_1,y_2,y_3,y_4,y_5,y_6,y_7,y_8,x_1,x_2,x_3,x_4,x_5,x_6,x_7,x_8),\\
\I_2 N &=(y_2,-y_1,-y_4,y_3,-y_6,y_5,y_8,-y_7,-x_2,x_1,x_4,-x_3,x_6,-x_5,-x_8,x_7),\\
\I_3 N &=(y_3,y_4,-y_1,-y_2,-y_7,-y_8,y_5,y_6,-x_3,-x_4,x_1,x_2,x_7,x_8,-x_5,-x_6),\\
\I_4 N &=(y_4,-y_3,y_2,-y_1,-y_8,y_7,-y_6,y_5,-x_4,x_3,-x_2,x_1,x_8,-x_7,x_6,-x_5),\\
\I_5 N &=(y_5,y_6,y_7,y_8,-y_1,-y_2,-y_3,-y_4,-x_5,-x_6,-x_7,-x_8,x_1,x_2,x_3,x_4),\\
\I_6 N &=(y_6,-y_5,y_8,-y_7,y_2,-y_1,y_4,-y_3,-x_6,x_5,-x_8,x_7,-x_2,x_1,-x_4,x_3),\\
\I_7 N &=(y_7,-y_8,-y_5,y_6,y_3,-y_4,-y_1,y_2,-x_7,x_8,x_5,-x_6,-x_3,x_4,x_1,-x_2),\\
\I_8 N &=(y_8,y_7,-y_6,-y_5,y_4,y_3,-y_2,-y_1,-x_8,-x_7,x_6,x_5,-x_4,-x_3,x_2,x_1),\\
\I_9 N &=(x_1,x_2,x_3,x_4,x_5,x_6,x_7,x_8,-y_1,-y_2,-y_3,-y_4,-y_5,-y_6,-y_7,-y_8).\\
\end{split}
\end{myequation2}
}

Their span,
\[
EN\ug<\I_1 N,\dots,\I_9 N>,
\]
is, at any point, a nine-plane in $\RR^{16}$ that is not tangential to the $S^{15}$ sphere. Observe that the nine-plane $EN$ is invariant under $\Spin{9}$. This is certainly the case for the single vector field $N$, since \mbox{$\Spin{9}\subset\SO{16}$}. On the other hand, the endomorphisms $\I_\alpha$ rotate under the $\Spin{9}$ action inside their $E^9 \subset \End{\RR^{16}}$ vector bundle.

Next, note that $EN$ contains $N$:
\[
N = \lambda_1 \I_1 N + \lambda_2 \I_2 N + \dots + \lambda_8 \I_8 N + \lambda_9 \I_9 N,
\]
where the coefficients $\lambda_\alpha$ are computed from (\ref{spin9B}) in terms of the inner products of vectors,
\[
\vec x =(x_1,\dots, x_8), \; \vec y=(y_1, \dots, y_8) \in \RR^8,
\]
and of the right translations, $R_i, \dots, R_h$, as follows:
\[
\lambda_1 = 2 \vec x \cdot \vec y, \qquad \lambda_2 =-2 \vec x \cdot \vec{R_i y}, \qquad \dots, \qquad \lambda_8 =-2 \vec x \cdot \vec{R_h y}, \qquad \lambda_9 = \vert \vec x \vert^2 - \vert \vec y \vert^2.
\]

In particular, at points with $\vec x = \vec 0$, that is, on the octonionic line $l_\infty$, the $\I_1 N, \dots, \I_9 N$ vector fields are orthogonal to the $S^7 \subset l_\infty$ unit sphere. The latter is the fiber of the Hopf fibration $S^{15} \rightarrow S^8$ over the north pole ($(0, \dots, 0,1) \in S^8$), and the mentioned orthogonality of this fiber ($S^7$) is immediate from~(\ref{spin9B}) for $\I_1 N, \dots, \I_8 N.$ 

Also, at these points, we have  $\I_9 N=N$, so  $\I_9N$ is orthogonal to the $S^7$ fiber as well. Now, the~invariance of the octonionic Hopf fibration under $\Spin{9}$ shows that all its fibers are characterized as being orthogonal to the vector fields $\I_1 N, \dots, \I_9 N$ in $\RR^{16}$. 

Now, assume that $X$ is a vertical vector field of $S^{15} \rightarrow S^8$. From the previous characterization, we~have the following orthogonality relations in $\RR^{16}$:
\[
\langle X, \I_\alpha N \rangle =0, \qquad\text{for }\alpha=1, \dots, 9,
\]
and it follows that  $\langle \I_\alpha X, N \rangle=0$. However, from the definition of a $\Spin{9}$ structure, it can be observed that if $\alpha \neq \beta$, then $\langle \I_\alpha X, \I_\beta X \rangle =0$. Thus, if $X$ is a nowhere zero vertical vector field, we~obtain, in this way, nine pairwise orthogonal vector fields ($\I_1 X, \dots, \I_9 X$) that are all tangent to $S^{15}$. However, $S^{15}$ is known to admit, at most, eight linearly independent vector fields. Thus, $X$ cannot be vertical and nowhere zero.
\end{proof}
One gets, as a consequence, the following alternative proof of a result, established in~Ref. \cite{LoVHFO}.

\begin{Corollary} \label{cor}
The octonionic Hopf fibration $S^{15} \rightarrow S^8$ does not admit any $S^1$ subfibration.
\end{Corollary}

\begin{proof}
In fact, any $S^1$ subfibration would give rise to a real line sub-bundle ($L\subset T_{\text{vert}}(S^{15})$) of the vertical sub-bundle of $T(S^{15})$. This line bundle ($L$) is necessarily trivial, due to the vanishing of its first Stiefel--Whitney class, $w_1(L)\in H^1(S^{15};\ZZ_2)=0$. It follows that $L$ would admit a nowhere zero section and thus, a global, vertical, nowhere zero vector field.
\end{proof}

\section{Locally conformally parallel $\Spin{9}$ manifolds}\label{se:lcp}

Let $G \subset \SO{d}$. Recall that a {\em locally conformally parallel $G$-structure} on a manifold $M^d$ is the datum of a Riemannian metric ($g$) on $M,$ a covering ${\mathcal U}=\{U_\alpha\}_{\alpha \in A}$ of $M,$ and for each $\alpha \in A$, a metric $g_\alpha$ defined on $U_\alpha$ which has holonomy contained in $G$ such that the restriction of $g$ to each $U_\alpha$ is conformal to $g_\alpha$:

$$g_{\vert U_\alpha}=e^{f_\alpha}g_\alpha$$
for some smooth map ($f_\alpha$) defined on $U_\alpha.$

Some of the possible cases here are

\begin{itemize}
\item $G=\U{n}$, where we have the {\em locally conformally K\"ahler metrics};

\item $G=\Sp{n}\cdot \Sp{1}$, yielding the {\em locally conformally quaternion K\"ahler metrics};

\item $G=\Spin{9},$ which is the case we are dealing with.
\end{itemize}

In any of the cases above, for each overlapping $U_\alpha\cap U_\beta$, the functions $f_\alpha, f_\beta$ differ by a constant: 
$$f_\alpha-f_\beta=ct_{\alpha, \beta}\,\, \text{on}\,\, U_\alpha\cap U_\beta.$$

This implies that  $df_\alpha=df_\beta$ on $U_\alpha \cap U_\beta\neq \emptyset$, hence defining a global, closed 1-form that is usually denoted by $\theta$ and called {\em the Lee form}.
Its metric dual with respect to $g$ is denoted by $N$ as
$$ N=\theta^{\sharp}$$ 
and is called  {\em the Lee vector field.}
 
The $G=\U{n}$ case of locally conformally K\"ahler metrics has been extensively studied in the last decades (see, for instance, Ref. \cite{DrOLCK}). 

When $G$ is chosen to be  $\Sp{n}$ or  $\Sp{n}\cdot\Sp{1}$, there are close relations to $3$-Sasakian geometry (see~Ref. \cite{OrPLCK} or the surveys~\cite{BoGTSM,CaPEWG}). Finally, locally conformally parallel $\Gtwo$ and $\Spin{7}$ structures were studied in~Ref. \cite{IPPLCP}, and they relate to nearly parallel $\SU{3}$ and $\Gtwo$ geometries, respectively. 

As mentioned in the Introduction, the holonomy of $\Spin{9}$ is only possible on manifolds that are either flat or locally isometric to $\OP{2}$ or to the hyperbolic Cayley plane $\OH{2}$. Weakened holonomy conditions give rise to several classes of $\Spin{9}$ structures (cf.~Ref. \cite{FriWSS} and Section \ref{Preliminaries}). One of these classes is that of \emph{vectorial type structures} (see~Refs. \cite{AgFGSV} and~\cite{FriWSS} (p.~148)). According to the following Definition and the following Remark, this class fits into the locally conformally parallel scheme.

\begin{Definition}\cite{AgFGSV}  \label{vectorial type}
A $\Spin{9}$ structure is of the \emph{vectorial type} if $\Gamma$ lives in $P_0$.
\end{Definition}

\begin{Remark}\label{re:victor}
In~Refs. \cite{FriWSS,AgFGSV}, the class of locally conformally parallel $\Spin{9}$ structures has been identified and studied under the name $\Spin{9}$ {\em structures of vectorial type}. 
Now, we outline now a proof that, for~$\Spin{9}$ structures, the vectorial type is equivalent to the locally conformally parallel type. As already mentioned in Section \ref{Preliminaries}, the splitting of the Levi--Civita connection, viewed as a connection in the principal bundle of orthonormal frames on $M$ is
\[
Z=Z^*\oplus\Gamma
\]
where $Z^*$ is the connection of $\Spin{9}$-frames in the induced bundle, and $\Gamma$ is its orthogonal complement. Thus, $\Gamma$ is a 1-form with values in the orthogonal complement $\lie{m}$ in the splitting of Lie algebras $\lieso{16}=\liespin{9}\oplus\lie{m}$, and under the identification $\Lambda^2_{84} = \lie{m} =\Lambda^3(E^9)$ (cf. the beginning of Section \ref{Canonical differential forms}), $\Gamma$ can be seen as a 1-form with values in $\Lambda^3(E^9)$. Under the action of $\Spin{9}$, the space $\Lambda^1(M)\otimes\Lambda^3(E)$ decomposes as a direct sum of four irreducible components:
\[
\Lambda^1(M)\otimes\Lambda^3(E)=P_0\oplus P_1\oplus P_2\oplus P_3,
\]
and looking at all the possible direct sums,  this yields 16 types of $\Spin{9}$ structure. Component $P_0$ identifies with $\Lambda^1(M)$, and thus, with the component $\mathcal W_4$ in Formula \eqref{W}.
\end{Remark}

Now, let $(M^{16},g)$ be a Riemannian manifold endowed with a $\Spin{9}$ structure of the vectorial type. Let $\Gamma$ be as above, and let $\snform$ be its $\Spin{9}$-invariant 8-form. Now, $\Gamma=0$ implies that the holonomy of $M$ is contained in $\Spin{9}$ (cf.~Ref. \cite{FriWSS} (p. 21)).

From~Ref. \cite{AgFGSV} (p. 5), we know that the following relations hold:
\begin{equation}\label{eq:friedrich}
d\snform=\theta\wedge\snform,\qquad d\theta=0.
\end{equation}

Let $(M,\tilde{g})$ be the Riemannian universal cover of $(M,g)$, and let $\tilde{\snform}$, $\tilde{\theta}$ be the lifts of $\snform$ and $\theta$ respectively. Then, relations ~\eqref{eq:friedrich} also hold for $\tilde{\snform}$ and $\tilde{\theta}$. Since $\tilde{M}$ is simply connected, $\tilde{\theta}=df$ for some $f:\tilde{M}\rightarrow\RR$. Then, by defining $g_0\ug e^{-f}\tilde{g}$ and $\snform_0\ug e^{-4f}\tilde{\snform}$, we have $d\snform_0=0$, that is, the $\theta$-factor of $\snform_0$ is zero. Hence, $g_0$ has holonomy contained in $\Spin{9}$, and on the other hand, it is locally conformal to $g$. Thus, $M$ can be covered by open subsets on which the metric is conformal to a metric with holonomy in $\Spin{9}$.

The conformal flatness of metrics with $\Spin{9}$ holonomy has the following consequences (cf.~Ref.~\cite{OPPSGO} for the proofs).
\begin{Theorem}\hypertarget{te:B}{}
Let $M^{16}$ be a compact manifold that is equipped with a locally, non globally, conformally parallel $\Spin{9}$ metric $g$. Then,
\begin{enumerate}
\item\label{te:st1} The Riemannian universal covering $(\tilde{M},\tilde{g})$ of $M$ is conformally equivalent to the euclidean ${\mathbb R}^{16}\setminus \{0\}$, the~Riemannian cone over $S^{15}$, and $M$ is locally isometric to $S^{15} \times \RR$ up to its homotheties.
\item\label{te:st2} $M$ is equipped with a canonical $8$-dimensional foliation.
\item\label{te:st3} If all the leaves of $\mathcal F$ are compact, then M fibers over the orbifold $\mathcal O^8$ are finitely covered by $S^8$, and all fibers are finitely covered by $S^7 \times S^1$.
\end{enumerate}
\end{Theorem}
\begin{Theorem}\hypertarget{te:C}{}
Let $(M,g)$ be a compact Riemannian manifold. Then, $(M,g)$ is locally, non globally, conformally parallel $\Spin{9}$ if and only if the following three properties are satisfied:
\begin{enumerate}
\item\label{it:structure1} $M$ is the total space of a fiber bundle $
M\stackrel{\pi}{\longrightarrow} S^1_r,$
where $\pi$ is a Riemannian submersion over a circle of radius $r$.
\item\label{it:structure2} The fibers of $\pi$ are spherical space forms ($S^{15}/K$), where $K$ is a finite subgroup of $\Spin{9}$.
\item\label{it:structure3} The structure group of $\pi$ is contained in the normalizer of $K$ in $\Spin{9}$.
\end{enumerate}
\end{Theorem}

\section{Clifford systems and Clifford structures}\label{Clifford}

The self dual anti-commuting involutions $\I_1, \dots , \I_9$ that define the standard $\Spin{9}$-structure on $\RR^{16}$ are an example of a {\em Clifford system}. The definition, formalized in 1981 by D. Ferus, H. Karcher, and~H. F. M\"unzner, in their study of isometric hypersurfaces of spheres \cite{FKMCNI}, is the following. 
\begin{Definition} A \emph{Clifford system} on the Euclidean vector space $\RR^N$ is the datum 
\[
C_m=(P_0,\dots , P_m)
\]
of a $(m+1)$-ple of symmetric endomorphisms $P_\alpha$ such that
\[
P_\alpha^2 = \; \Id \; \; \text{for all} \; \; \alpha, \qquad P_\alpha P_\beta = -P_\beta P_\alpha \; \; \text{for all} \; \; \alpha \neq \beta.
\]

A Clifford system on $\RR^N$ is said to be \emph{irreducible} if $\RR^N$ is not a direct sum of two positive dimensional subspaces that are invariant under all the $P_\alpha$.
\end{Definition}
From the representation theory of Clifford algebras, it is recognized (cf. Refs. \cite{FKMCNI} (p. 483) and~\cite{HusFiB}~(p.~163)) that $\RR^N$ admits an irreducible Clifford system ($C=(P_0,\dots , P_m)$) if and only if 
$N= 2\delta (m)$,
where $\delta(m)$ is given as in the following Table \ref{tab2}. 

Uniqueness can be discussed as follows. Given, on $\RR^N$, two Clifford systems ($C_m=(P_0,\dots , P_m)$ and $C'_m=(P'_0,\dots , P'_m)$), they are said to be \emph{equivalent} if $A \in O(N)$ exists such that $P'_\alpha = A^t P_\alpha A$  for all $\alpha$. Then, for $m \not\equiv 0$ mod $4$, there is a unique equivalence class of irreducible Clifford systems, and for $m \equiv 0$ mod $4$, there are two, which are classified by the two possible values of the trace $tr\; ({P_0 P_1 \dots P_m}) = \pm 2 \delta (m)$.

\begin{table}[H]
\caption{Clifford systems.}\label{tab2}
\begin{center}
{\begin{tabular}{cccccccccccccccccccccc}
\toprule
$m$&$1$&$2$&$3$&$4$&$5$&$6$&$7$&$8$&$9$&$10$&$11$&$12$&$13$&$14$&$15$&$16$&$\dots$&$8+h$\\
\midrule
$\delta(m)$&$1$&$2$&$4$&$4$&$8$&$8$&$8$&$8$&$16$&$32$&$64$&$64$&$128$&$128$&$128$&$128$&$\dots$&$16\delta(h)$\\
\bottomrule
\end{tabular}}
\end{center}
\end{table}

In Ref.~\cite{PPVCSO}, we outlined the following inductive construction for the irreducible Clifford systems on real Euclidean vector spaces ($\RR^N$), taking, as starting the point, the basic Clifford systems ($C_1,C_2,C_4,C_8$) associated with structures given by  $\U{1}, \U{2}, \Sp{2}\cdot \Sp{1}$, and $\Spin{9}$. All the cases appearing in Tables \ref{tab2} and \ref{tab3} make sense in the natural context of Riemannian manifolds. We get the following theorem (see Ref.~\cite{PPVCSO} for details).

\begin{Theorem}\label{Procedure} {\rm(Procedure to write new Clifford systems from old)}. Let $C_m= (P_0, P_1, \dots , P_m)$ be the last (or unique) Clifford system in $\RR^N$. Then, the first (or unique) Clifford system, 
\[
C_{m+1}=(Q_0, Q_1, \dots , Q_m, Q_{m+1}),
\]  
in $\RR^{2N}$ has, respectively, the following first and last endomorphisms:
\[
Q_0 =\left(
\begin{array}{r|r}
0 & \Id \\
\hline 
\Id & 0
\end{array}\right), \quad Q_{m+1}=\left(
\begin{array}{r|r}
\Id & 0\\
\hline 
 0 & -\Id
\end{array}\right),
\]
where the blocks are $N \times N$. The remaining matrices are
\[
\qquad   \qquad  \qquad  \qquad Q_\alpha =\left(
\begin{array}{c|c}
0 & -P_{0\alpha} \\
\hline 
P_{0 \alpha} & 0
\end{array}\right) \qquad  \qquad  \qquad  (\alpha=1,\dots, m). 
\] 

Here, $P_{0 \alpha}$ are the complex structures given by $P_0 P_\alpha$ compositions in the Clifford system $C_m$. When the complex structures ($P_{0\alpha}$) can be viewed as (possibly block-wise) right multiplications by some of the unit quaternions ($i,j,k$) or unit octonions ($i,j,k,e,f,g,h$), and if the dimension permits it, further similarly defined endomorphisms ($Q_\beta$) can be added by using some others among $i,j,k$ or $i,j,k,e,f,g,h$.
\end{Theorem}

\begin{table}[H]
\caption{Clifford systems $C_m$ and $G$-structures on Riemannian manifolds ($M^N$).}\label{tab3}
\begin{center}
\scalebox{0.75}[0.75]{
{
\begin{tabular}{cccccccccccccccc}
\toprule 
$m$&$1$&$2$&$3$&$4$&$5$&$6$&$7$&$8$&$9$&$10$&$11$&$12$\\
\midrule
$N$&$2$&$4$&$8$&$8$&$16$&$16$&$16$&$16$&$32$&$64$&$128$&$128$\\\midrule
$G$&$\U{1}$&$\U{2}$&$\Sp{1}^3$&$\Sp{2}\Sp{1}$&$\SU{4}\Sp{1}$&$\Spin{7}\U{1}$&$\Spin{8}$&$\Spin{9}$&$\Spin{10}$&$\Spin{11}$&$\Spin{12}$&$\Spin{13}$\\
\bottomrule
\end{tabular}
}
}
\end{center}
\end{table}

The notion of an even Clifford structure, a kind of unifying notion proposed by A Moroianu and U. Semmelmann \cite{MoSCSR}, is instead given by the following datum on a Riemannian manifold ($(M,g)$).

\begin{Definition} An \emph{even Clifford structure} on $(M,g)$ is a real oriented Euclidean vector bundle $(E,h)$, together with an algebra bundle morphism ($\varphi: \; \text{Cl}^0(E) \rightarrow \End{TM}$) which maps $\Lambda^2 E$ into skew-symmetric endomorphisms. 
\end{Definition}

By definition, a Clifford system always gives rise to an even Clifford structure, but there are some even Clifford structures on manifolds that cannot be constructed, even locally, from Clifford systems. An example of this is given by a $\Spin{7}$ structure on any oriented 8-dimensional Riemannian manifold as a consequence of the following observations (cf. Ref. \cite{PaPSAC} for further details).

\begin{Proposition}
Let $C_m=\{P_0,\dots,P_m)$ be a Clifford system in $\RR^n$. The compositions $\J_{\alpha\beta}\ug P_\alpha P_\beta$ for $\alpha<\beta$, and $\J_{\alpha\beta\gamma}\ug P_\alpha P_\beta P_\gamma$ for $\alpha<\beta<\gamma,$ are linearly independent complex structures on $\RR^n$.
\end{Proposition}

\begin{proof}
It can be easily recognized that $\J_{\alpha\beta}$ and $\J_{\alpha\beta\gamma}$ are complex structures. On the other hand, for~any $\alpha=0,\dots,m$, it can be observed that $tr (P^*_\alpha P_\alpha)=1$, and for $\alpha <\beta$, $tr (P^*_\alpha P_\beta)=0$, 
so that the $P_\alpha$ are orthonormal and symmetric. By a similar argument, $ tr (\J^*_{\alpha\beta}\J_{\alpha\beta})=1$ and $ tr (\J^*_{\alpha\beta}\J_{\gamma\delta})= tr(P_\beta P_\alpha P_\gamma P_\delta)=0$ if $\gamma$ or $\delta$ 
equals $\alpha$ or $\beta$. Also, for $\alpha\neq\gamma$ and $\beta\neq\delta$, the $\J^*_{\alpha\beta}\J_{\gamma\delta}$ is the composition of the skew-symmetric $\J_{\beta\alpha\gamma}$ and the symmetric $P_\delta$, and as such, its trace is necessarily zero. Similar arguments show that the $\J_{\alpha\beta\gamma}$ for $\alpha<\beta<\gamma$ are orthonormal.
\end{proof}

\begin{Corollary}\label{spin7}
The $\Spin{7}$-structures on $\RR^8$ cannot be defined through Clifford system $C_6$.
\end{Corollary}

\begin{proof}
For any choice of such a Clifford system, $C_6=(P_0, \dots P_6)$ in $\RR^8$, the complex structures $\J_{\alpha\beta\gamma}$ for $\alpha<\beta<\gamma$ give rise to $35$ linearly independent skew-symmetric endomorphisms, contradicting the decomposition of $2$-forms in $\RR^8$ under $\Spin{7}$:
\begin{equation}\label{dec7'}
\Lambda^2\RR^8 = \Lambda^2_7 \oplus \Lambda^2_{21}.
\end{equation}
\end{proof}

Nevertheless, the right multiplications by $i,j,k,e,f,g,h \in\OO$ span the $E^7 \subset \mathrm{End}^- \RR^8$ vector bundle, and this identifies $\Spin{7}$ structures among the even Clifford structures.

The following Sections present further examples of such \emph{essential Clifford structures}, i.e., Clifford structures not coming from Clifford systems.

On Riemannian manifolds ($(M,g)$), it is natural to consider the following class of even Clifford~structures.

\begin{Definition} The even Clifford structure $E^r$ on $(M,g)$ is said to be \emph{parallel} if a metric connection ($\nabla^E$) exists on $E$ such that $\varphi$ is connection preserving, i.e.,
\[
\varphi(\nabla^E_X \sigma) = \nabla^g_X \varphi (\sigma)
\]
for every tangent vector $X \in TM$ and section $\sigma$ of $\text{Cl}^0 E$, and where $\nabla^g$ is the Levi--Civita connection. 
\end{Definition}

Table \ref{tab4} summarizes the non-flat, parallel, even Clifford structure, as classified in Ref. \cite{MoSCSR}. The~non-compact duals of the appearing symmetric spaces have to be added. A good part of the listed manifolds appear in the following sections.

\begin{table}[H]
\caption{Parallel, non-flat, even Clifford structures (cf. Ref. \cite{MoSCSR}).}\label{tab4}
\begin{center}
\scalebox{0.85}[0.85]{
{
\begin{tabular}{cccc}
\toprule
\boldmath$r$ & \textbf{Type of} \boldmath{$E^r$} &\boldmath$M$&\textbf{Dimension of} \boldmath{$M$}\\
\midrule
2&&K\"ahler&$2m, m\geq 1$\\

3&projective if $M \neq \HH P^q$&quaternion K\"ahler (qK) &$4q, q \geq 1$ \\

4&projective if $M \neq \HH P^{q^+} \times \HH P^{q^-}$&product of two qK&$4(q^++q^-) $\\
\midrule
5&&qK&8\\

6&projective if $M$ non-spin&K\"ahler &8 \\

7&&$\Spin{7}$ holonomy& 8\\

8&projective if $M$ non-spin&Riemannian&8\\
\midrule
5& &$Gr_{2}(\HH^{n+2})$&$8n$\\

6&projective for $n$ odd&$Gr_{4}(\CC^{n+4})$&$8n$\\

8&projective for $n$ odd &$Gr_{8}(\RR^{n+8})$&$8n$ \\
\midrule
9& &$\FII$&16 \\

10& &$\EIII$& 32 \\

12& &$\EVI$&64 \\

16& &$\EVIII$& 128 \\
\bottomrule
\end{tabular}
}}
\end{center}
\end{table}

\section{The complex Cayley projective plane}\label{Complex Cayley}

This section deals with 
$$\EIII = \mathrm{E_6}/\Spin{10}\cdot \U{1} \cong  V_{16}^{78} \subset \CC P^{26},$$ 
the second, after the Cayley projective plane ($\FII$), of the exceptional symmetric spaces appearing in Table \ref{tab4}.

A remarkable feature of $\EIII$, one of the two exceptional Hermitian symmetric spaces of compact type, is its model as a smooth projective algebraic variety of complex dimension $16$ and degree $78$, the so-called fourth Severi variety $ V_{16}^{78} \subset \CC P^{26}$. This name was proposed by F. Zak \cite{ZakSeV}, who classified the smooth projective algebraic varieties ($V_n$) in a $\CC P^N$ that, in spite of their critical dimension ($n=\frac{2}{3}(N-2)$), are unable to fill $\CC P^N$ through their chords.

On the other hand, $\EIII$ admits a construction that is very similar to the one of the Cayley projective plane ($\FII$). One can, in fact, look at the complex octonionic Hermitian matrices
$$Z= \left(
\begin{array}{rrr}
c_1& x_1 &x_2 \\
\bar x_1  & c_2 & x_3 \\
\bar x_2 & \bar x_3 & c_3
\end{array}
\right) \in \mathrm{Herm}_3  (\CC \otimes \OO) \equiv \CC^{27}, \; \; c_\alpha \in \CC; \; x_\alpha \in \CC \otimes \OO,$$ 
which are acted on by $\mathrm{E}_6$ with three orbits on $\CC P^{26}$. The closed one consists of $Z$ matrices of rank one,
\[
Z^2 = (\mathrm{trace} \; Z) Z,
\]
and, as such, can be thought as (virtual) ``projectors on complex octonionic lines in $(\CC \otimes \OO)^3$''; thus, they are points of the complex projective Cayley plane $\EIII = \mathrm{E_6}/\Spin{10}\cdot \U{1} \subset \CC P^{26}$.

The projective algebraic geometry of $\EIII \subset \CC P^{26}$ was studied in detail in Ref. \cite{IlMCRC}. Similarly to Corollary \ref{spin7}, we have the following proposition. 

\begin{Proposition}\label{pr:ind}
The complex space $\CC^{16}$ does not admit any family of ten endomorphisms ($P_0,\dots,P_9$) that satisfies the properties of a Clifford system and is compatible with respect to the standard Hermitian scalar product $g$.
\end{Proposition}

\begin{proof} The family $P_0, \dots , P_9$ would define (after multiplying each of them by $i$) a representation of the complex Clifford algebra $\CC \mathrm{l}_{10}\cong \CC(32)$ (the order $32$ complex matrix algebra) on the $\CC^{16}$ vector space.
\end{proof}

Note, however, that the Euclidean space $\RR^{32}$ admits the Clifford system $C_9$ (cf. Table \ref{tab2}). The~parallel, even Clifford structure on $\EIII$ can be defined through the following one, here described on the $\CC^{16}$ model space. 

For this, observe that $\mathfrak{spin}(10) \subset \mathfrak{su}(16)$ is generated as a Lie algebra by $\mathfrak{spin}(9)$ and $\mathfrak{u}(1)$, with~$\mathfrak{u}(1)$ spanned by
\[ 
\begin{pmatrix}i\Id_8 & 0 \\ 0 & -i\Id_8\end{pmatrix}=\begin{pmatrix}
i\Id_8 & 0 \\
0 & i\Id_8
\end{pmatrix} \cdot \begin{pmatrix}
\Id_8 & 0 \\
0 & -\Id_8
\end{pmatrix}
= \mathcal I_0 \cdot \mathcal I_9 =\J_{09}
 \]
\text{where} $ \I_0 = \mathfrak I\;  \text{is a complex structure of} \; \CC^{16}, \; \text{and} \;  \I_1, \dots , \I_9 \; \text{are the octonionic Pauli matrices} .
$

The rank 10 even Clifford structure on $\CC^{16}$ is then given by the vector bundle
\[
E^{10}= <\mathcal I_0 >\oplus \; <\I_1,\dots,\I_9> =<\mathfrak I >\oplus \; <\I_1,\dots,\I_9>\subset\End{TM},
\]
and
\[
\liespin{10}= \mathfrak{lie}\{\J_{09}, \J_{19}, \dots, \J_{89}\} = \; \text{span}\{\J_{\alpha\beta}=\I_\alpha\circ\I_\beta\}_{0\le\alpha<\beta\le9}.
\]

We get the following theorem (cf. Proposition \ref{pr:charpoly} and Theorem \ref{teo:main}).

\begin{Theorem}[\cite{PaPECS}]
Let $E^{10}$ be the even Clifford structure on $\EIII$. In accordance with the previous notations, the~characteristic polynomials 
\[
t^{10}+\tau_2 (\psi) t^8 + \tau_4 (\psi)t^6 + \dots
\]  
of the matrix $$\psi=(\psi_{\alpha \beta}) \in \Lambda^2 \otimes \mathfrak{so}(10)$$ of K\"ahler 2-forms of the $\J_{\alpha\beta}$ give:
\begin{enumerate}
\item[(i)] $\tau_2 (\psi) =  -3\omega^2
\; \text{where} \; \omega\; \text{is the K\"ahler 2-form of} \; \EIII $;
\item[(ii)] $[\tau_4 (\psi)] \in H^8$ is the primitive generator of the cohomology ring $H^*(\EIII;\RR)$.
\end{enumerate}
\end{Theorem}

In analogy with the $\Spin{9}$ situation (Section \ref{Canonical differential forms}), $\tau_4 (\psi)=\form{\Spin{10}\cdot \U{1}}$ is called the \emph{canonical 8-form} on $\EIII$.

Moreover, the following theorem applies. 

\begin{Theorem}\label{main} Let $\omega$ be the K\"ahler form, and let $\Phi_{\Spin{10}} = \tau_4(\psi)$ be the previously defined 8-form on $\EIII$. Then:
\begin{enumerate}
\item[(i)]  The de Rham cohomology algebra $H^*(\EIII)$ is generated by (the classes of) $\omega \in \Lambda^2$ and $\Phi_{\Spin{10}} \in \Lambda^8$.

\item[(ii)]  By looking at $\EIII$ as the fourth Severi variety ($V_{16}^{78} \subset \CC P^{26}$), the de Rham dual of the basis represented in $H^8(\EIII; \ZZ)$ by the forms $(\frac{1}{(2\pi)^4} \Phi_{\Spin{10}}, \frac{1}{(2\pi)^4} \omega^4)$ is given by the pair of algebraic cycles
\[
\Big(\CC P^4 + 3(\CC P^4)', \; \;  \CC P^4 + 5(\CC P^4)' \Big),
\]
where $\CC P^4, (\CC P^4)'$ are maximal linear subspaces that belong to the two different families ruling a totally geodesic, non-singular, quadric $Q_8$ contained in $V_{16}^{78}$.
\end{enumerate}
\end{Theorem}

\section{Cayley-Rosenfeld planes}\label{Rosenfeld}

Besides the real and the complex Cayley projective planes $\FII$ and $\EIII$, there are two further exceptional symmetric spaces of compact type, that are usually referred to as the \emph{Cayley-Rosenfeld projective planes}, namely, the ``projective plane over the quaternionic octonions'':
$$\EVI=\mathrm{E_7}/\Spin{12}\cdot \Sp{1}=(\HH \otimes \OO) P^2,$$
and the ``projective plane over the octonionic octonions'':
$$\EVIII= \mathrm{E_8}/\Spin{16}^+ = (\OO \otimes \OO) P^2.$$

By referring to the inclusions $\OO \hookrightarrow \CC \otimes \OO \hookrightarrow \HH \otimes \OO \hookrightarrow \OO \otimes \OO$, it can be observed that the projective geometry of these four projective planes, which is notably present on the first two steps, becomes weaker at the third and fourth Cayley-Rosenfeld projective planes (cf. Ref. \cite{BaeOct}). However, the~dimensions of these four exceptional symmetric spaces are coherent with this terminology. According to Table \ref{tab4}, these four spaces have the highest possible ranks for non-flat, parallel, even~Clifford structures. 

The following Table \ref{tab5} summarizes the even Clifford structures on the four Cayley-Rosenfeld planes. Only the one on $\FII$ is given by a Clifford system; the other three are essential. Concerning the cohomology generators, the one in dimension $8$ can also be constructed for $\EVI$ and $\EVIII$ via the fourth coefficient ($\tau_4(\psi)$) of the matrices ($\psi$) of K\"ahler 2-forms that are associated with the even Clifford structures that we are now listing.

\begin{table}[H]
\caption{Even Clifford structures on the Cayley-Rosenfeld projective planes.}\label{tab5}
\begin{center}
{
\begin{tabular}{cccc}
\toprule
\textbf{Model }& \textbf{Symm. Space} & \textbf{Even Clifford Structure}  &  \textbf{Cohomology Gen.}\\
\midrule
$\RR^{16}$&$\FII$  & $E^9=<\I_1,\dots,\I_9>$   &  in $H^8$\\
$\CC^{16}$&$\EIII$& $E^{10}<\mathfrak I >\oplus <I_1,\dots,\I_9>$ &  in $H^2, H^8$\\
$\HH^{16}$&$\EVI$&$E^{12}=<\mathfrak I , \mathfrak J , \mathfrak K>\oplus <\I_1,\dots,\I_9>$ &  in $H^4, H^8, H^{12}$ \\
$\OO^{16}$ &$\EVIII$ & $E^{16}=<\mathfrak I , \dots, \mathfrak H>\oplus <\I_1,\dots,\I_9>$ &   in $H^8, H^{12}, H^{16}, H^{20}$ \\
\bottomrule 
\end{tabular}
}
\end{center}
\end{table}

\begin{Remark}
The matrices of the K\"ahler 2-forms that are associated with the even Clifford structures go through
\[
\liespin{9} \subset \liespin{10} \subset \liespin{12} \subset \liespin{16},
\]
and the last Lie algebra decomposes as
\[
\liespin{16}= \mathfrak{so}(16)= \liespin{9} \oplus \Lambda^2_{84}.
\]
By recalling the identification $(1\leq \alpha<\beta<\gamma \leq 9)$
\[
\Lambda^2_{84} = <J_{\alpha\beta\gamma}=\I_\alpha \I_\beta \I_\gamma>
\]
(cf. the observation after \eqref{decomposition} as well as \cite{PPVCSO}), this identification takes back the even Clifford structure of the $128$-dimensional Cayley-Rosenfeld plane ($\EVIII$) to the $\Spin{9}$-structures that we started with.
\end{Remark}

\section{Exceptional symmetric spaces}\label{Exceptional}

 A good number of symmetric spaces that appear in Table \ref{tab4} belong to the list of exceptional Riemannian symmetric spaces of compact type, 
\[
\EI, \, \EII, \; \EIII, \; \EIV, \; \EV, \; \EVI, \; \EVII, \; \EVIII, \; \EIX, \; \FI, \; \FII,  \; \GI, 
\] 
that are part of the E. Cartan classification. Among them, the two exceptional Hermitian symmetric spaces,
\[
\EIII = \frac{\Esix}{\mathrm{Spin}(10)\cdot\mathrm{U}(1)} \quad \text{and} \quad \EVII =  \frac{\Eseven}{\Esix \cdot\mathrm{U}(1)},
\]
are K\"ahler and therefore, are equipped with a non-flat, parallel, even Clifford structure of rank $r=2$. Next, the five Wolf spaces
\[
\begin{split}
\EII= \frac{\Esix}{\mathrm{SU}(6)\cdot\mathrm{Sp}(1)},\quad  \EVI=  \frac{\Eseven}{\mathrm{Spin}(12)\cdot\mathrm{Sp}(1)}, \quad \EIX =  \frac{\Eeight}{\Eseven\cdot\mathrm{Sp}(1)}, \\
\FI =  \frac{\Ffour}{\mathrm{Sp}(3)\cdot\mathrm{Sp}(1)}, \quad  \GI = \frac{\Gtwo}{\mathrm{SO}(4)}, \qquad \qquad \qquad \qquad 
\end{split}
\]
which are examples of positive quaternion K\"ahler manifolds, carry a rank of $r=3$ in a non-flat, parallel, Clifford structure.

Thus, seven of the twelve exceptional, compact type, Riemannian symmetric spaces of are either K\"ahler or quaternion K\"ahler. Accordingly, one of their de Rham cohomology generators is represented by a K\"ahler or quaternion K\"ahler form, and any further cohomology generators can be viewed as primitive in the sense of the Lefschetz decomposition.

As seen in the previous sections, the four Cayley-Rosenfeld projective planes, $$\EIII, \quad \EVI, \quad \EVIII , \quad \FII$$ carry a similar structure with $r=10,12,16,9$. 

Thus, among the exceptional symmetric spaces of compact type, there are two spaces that admit two distinct even Clifford structures, namely, the Hermitian symmetric $\EIII$ has even Clifford structures of rank $2$ and of rank $10$, and the quaternion K\"ahler $\EVI$ has even Clifford structures of rank $3$ and rank of $12$. For simplicity, we call \emph{octonionic K\"ahler} the parallel even Clifford structure defined by the vector bundles $E^{10},E^{12},E^{16},E^9$ on the Cayley-Rosenfeld projective planes ($\EIII, \, \EVI, \, \EVIII , \, \FII$). In~conclusion, and with the exceptions of $$\EI = \frac{\Esix}{\Sp{4}}, \quad \EIV= \frac{\Esix}{\Ffour},\quad \EV=\frac{\Esix}{\SU{8}},$$ nine of the twelve exceptional Riemannian symmetric spaces of compact type admit at least one parallel, even Clifford structure. Any of such structures gives rise to a canonical differential form: the~K\"ahler 2-form $\omega$ for the complex K\"ahler one, the quaternion K\"ahler 4-form $\Omega$ for the five Wolf space, \mbox{and a canonical} octonionic K\"ahler 8-form $\Psi$ for the four Cayley-Rosenfeld projective planes. Their~classes are always one of the cohomology generators, and Table \ref{tab6} collects some informations on the exceptional symmetric spaces of compact type. For each of them, the real dimension, the existence of torsion in the integral cohomology, the K\"ahler or quaternion K\"ahler or octonionic K\"ahler (K/qK/oK) property, the Euler characteristic $\chi$, and the Poincar\'e polynomial (up to mid dimension) are listed. 

\begin{table}[H]
\caption{Exceptional, compact type, symmetric spaces.}\label{tab6}
\begin{center}
\scalebox{0.7}[0.7]{
{
\begin{tabular}{cccccc}
\toprule
 & \textbf{dim} &\textbf{torsion}&\textbf{K/qK/oK}&\boldmath$\chi$&\textbf{Poincar\'e polynomial }\boldmath{$P(t)=\sum_{i=0, \dots} b_i t^i$}\\
\midrule
$\EI$&42&yes&&4&$1+t^8+t^9+t^{16}+t^{17}+t^{18}+ \dots $\\
$\EII$&40&yes &qK &36& $1+t^4+t^6+2t^{8}+t^{10}+3t^{12}+ 2t^{14}+3t^{16}+2t^{18}+4t^{20}+\dots $\\
$\EIII$&32& no& K/oK &27&  $1+t^2+t^4+t^6+2(t^{8}+t^{10}+t^{12}+ t^{14})+3t^{16}+\dots $\\
$\EIV$&26& no&&0&$1+t^9+\dots$\\
$\EV$&70&yes &&72&$1+t^6+t^8+t^{10}+t^{12}+2(t^{14}+ t^{16}+t^{18}+ t^{20})+3(t^{22}+t^{24}+t^{26}+t^{28})+4(t^{30}+t^{32})+3t^{34}+\dots $ \\
$\EVI$&64& yes& qK/oK&63&$1+t^4+2t^8+3t^{12}+4t^{16}+5t^{20}+6(t^{24}+t^{28})+7t^{32}+\dots $ \\
$\EVII$&54&no&K&56&$1+t^2+t^4+t^6+t^8+2(t^{10}+t^{12}+t^{14}+ t^{16})+3(t^{18}+ t^{20}+t^{22}+t^{24}+t^{26})+\dots $\\
$\EVIII$&128 &yes&oK&135&$1+t^8+t^{12}+2(t^{16}+ t^{20})+3(t^{24}+t^{28})+5t^{32}+4t^{36}+6(t^{40}+t^{44})+7(t^{48}+t^{52})+8t^{56}+7t^{60}+9t^{64}+\dots $\\
$\EIX$&112& yes&qK& 120&$1+t^4+t^8+2(t^{12}+t^{16})+3 t^{20}+4(t^{24}+t^{28})+5t^{32}+6(t^{36}+t^{40})+7(t^{44}+t^{48}+t^{52})+8t^{56}+\dots $\\
$\FI$&28 & yes&qK&12&$1+t^4+2(t^8+t^{12})+\dots$ \\
$\FII$&16 &no&oK&3&$1+t^8+\dots$ \\
$\GI$&8 & yes& qK&3&$1+t^4+\dots$ \\
\bottomrule
\end{tabular}
}}
\end{center}
\end{table}

Next, we have Table \ref{tab7} which contains the primitive Poincar\'e polynomials $$\widetilde{P}(t)=\sum_{i=0, \dots} \widetilde{b}_i t^i$$ of the nine exceptional Riemannian symmetric spaces that admit an even parallel Clifford structure. Here, the meaning of ``primitive'' varies depending on the considered  K/qK/oK structure.  Thus, the~Hermitian symmetric spaces, $\EIII$ and $\EVII$, are simply polynomials with coefficients the primitive Betti numbers, $$\widetilde{b}_i  = \dim \; (\ker [L_\omega^{n-i+1}: H^i \rightarrow H^{2n-i+2}]),$$ where $L_\omega$ is the Lefschetz operator which multiplies the cohomology classes with the complex K\"ahler form $\omega$, and $n$ is the complex dimension. 

In the positive quaternion K\"ahler setting, the vanishing of odd Betti numbers and the injectivity of the Lefschetz operator $L_\Omega: H^{2k-4} \rightarrow H^{2k}$, $k \leq n$, now occur with $\Omega$ being the quaternion 4-form and $n$ being the quaternionic dimension. A remarkable aspect of the primitive Betti numbers $$\widetilde{b}_{2k} = \dim ( \text{coker}  [L_\Omega: H^{2k-4} \rightarrow H^{2k}])$$ for positive quaternion K\"ahler manifolds is their coincidence with the ordinary Betti numbers of the associated Konishi bundle---the 3-Sasakian manifold fibering over it (cf. ~\cite{GaSBNT} (p. 56)).  

Finally, on the four Cayley-Rosenfeld planes, the vanishing of odd Betti numbers and the injectivity of the map $L_\Phi: H^{2k-8} \rightarrow H^{2k}$ still occur, and are defined by multiplication with the octonionic 8-form $\Phi$, and with $k \leq 2n$, where $n$ is now the octonionic dimension.

\begin{table}[H]
\caption{Primitive Poincar\'e polynomials, $\widetilde{P}(t)=\sum_{i=0, \dots} \widetilde{b}_i t^i$.}\label{tab7}
\begin{center}
\scalebox{0.85}[0.85]{
{
\begin{tabular}{cc}
\toprule
\textbf{Hermitian Symmetric Spaces}&\textbf{K\"ahler Primitive Poincar\'e Polynomial}\\
\midrule
$\EIII$&  $1+t^{8}+t^{16}$\\
$\EVII$&$1+t^{10}+t^{18}$\\
\midrule
 Wolf spaces& Quaternion K\"ahler primitive Poincar\'e polynomial \\\midrule
$\EII$& $1+t^6+t^{8}+t^{12}+ t^{14}+t^{20} $\\
$\EVI$&$1+t^8+t^{12}+t^{16}+t^{20}+t^{24}+t^{32}$\\
$\EIX$&$1+t^{12}+t^{20}+t^{24}+t^{32}+t^{36}+t^{44}+t^{56} $\\
$\FI$&$1+t^8$\\
$\GI$&$1$ \\
\midrule
 Cayley-Rosenfeld projective planes&Octonionic K\"ahler primitive Poincar\'e polynomial \\
\midrule
$\EIII$& $1+t^2+t^4+t^6+t^{8}+t^{10}+t^{12}+ t^{14}+t^{16}$\\
$\EVI$&$1+t^4+t^8+2(t^{12}+t^{16}+t^{20})+3(t^{24}+t^{28}+t^{32})$ \\
$\EVIII$&$1+t^{12}+t^{16}+ t^{20}+t^{24}+t^{28}+t^{32}+t^{36}+t^{40}+t^{44}+t^{48}+t^{52}+t^{56}+t^{60}+t^{64}$\\
$\FII$&$1$ \\
\midrule
Even Clifford exceptional symmetric spaces &Fully primitive Poincar\'e polynomial \\
\midrule
$\EII$& $1+t^6+t^{8}+t^{12}+ t^{14}+t^{20} $\\
$\EIII$&  $1 $\\
$\EVI$&$1+t^{12}+t^{24}$\\
$\EVII$&$1+t^{10}+t^{18}$\\
$\EVIII$&$1+t^{12}+t^{16}+ t^{20}+t^{24}+t^{28}+t^{32}+t^{36}+t^{40}+t^{44}+t^{48}+t^{52}+t^{56}+t^{60}+t^{64}$\\
$\EIX$&$1+t^{12}+t^{20}+t^{24}+t^{32}+t^{36}+t^{44}+t^{56} $\\
$\FI$&$1+t^8$\\
$\FII$&$1$ \\
$\GI$&$1$ \\
\bottomrule
\end{tabular}
}}
\end{center}
\end{table}

\section{Grassmannians}\label{Grassmannians}
Table \ref{tab4} contains the following three series of Grassmannians:
\begin{equation}\label{Gra}
Gr_{8}(\RR^{n+8})=\frac{\mathrm{SO}(n+8)}{\mathrm{SO}(n)\times \mathrm{SO}(8)}, \; \; Gr_{4}(\CC^{n+4})=\frac{\mathrm{SU}(n+4)}{\mathrm{S}(\mathrm{U}(n)\times \mathrm{U}(4))}, \;  \;Gr_{2}(\HH^{n+2})=\frac{\mathrm{Sp}(n+2)}{\mathrm{Sp}(n)\times \mathrm{Sp}(2)},
\end{equation}
that carry an even Clifford structure of rank $r=8,6,5$, respectively.

To define them, recall that the \, $\Spin{8} \subset \SO{16} \subset \mathrm{Cl} \, \OO$ \; subgroup is generated by the following~matrices:
\begin{equation}\label{muv}
m_{u,v}= \left(
\begin{array}{cc}
-R_u \circ R_{\bar v}& 0 \\
0& -R_{\bar u} \circ R_v
\end{array}
\right) = m_u \circ m_v, 
\end{equation}
where 
\[
m_u= \left(
\begin{array}{cc}
0& R_u \\
-R_{\bar u}& 0
\end{array}
\right) , \; \; m_v= \left(
\begin{array}{cc}
0& R_v \\
-R_{\bar v}& 0
\end{array}
\right)
\]
(cf. Ref. \cite{BryRSL}). For the orthonormal $u,v \in S^7 \subset \OO$, matrices $m_{u,v}$ satisfy the properties
$$m_{v,u} =-m_{u,v}, \qquad m_{u,v}^2 = - \Id.$$

On the other hand, recall that
$$ T Gr \cong W \otimes W^\perp, $$
where $W$ is the tautological vector bundle, and $W^\perp$ its orthogonal complement in the ambient linear~space. 

Finally, it is convenient to recall that the complex structure and the local compatible hypercomplex structures of the following complex K\"ahler and quaternion K\"ahler Grassmannians,

\begin{equation*}
Gr_{2}(\RR^{n+2})\cong Q_n \subset \CC P^{n+1}, \qquad Gr_{4}(\RR^{n+4})=\frac{\mathrm{SO}(n+4)}{\mathrm{SO}(n)\times \mathrm{SO}(4)}, \qquad Gr_{2}(\CC^{n+2})=\frac{\mathrm{SU}(n+2)}{\mathrm{S}(\mathrm{U}(n)\times \mathrm{U}(2))},
\end{equation*}
come visibly from their elements, which are, respectively, oriented $2$-planes, oriented $4$-planes, and complex $2$-planes.

Look first at Grassmannians in the first of the three series \eqref{Gra}, namely, at $Gr_{8}(\RR^{2m+8})$, referring, for simplicity, to the case of an even dimensional ambient space, $\RR^{2m+8}$, and thus insuring the spin property of the Grassmannian. From the local orthonormal bases $w_1 \dots , w_8$ and $w_1^\perp , w_2^\perp , \dots , w_{2m-1}^\perp, w_{2m}^\perp$
of sections respectively of $W$ and of $W^\perp$, one gets the following local basis of the tangent vectors of $Gr_{8}(\RR^{2m+8})$:
\begin{equation}\label{thex}
\begin{aligned}
&x_{1,1}=w_1 \otimes w_1^\perp , &\dots \dots \;  & &\quad x_{8,1}= w_8 \otimes w_1^\perp ,\\
&x_{1,2}=w_1 \otimes w_2^\perp , &\dots \dots \;  & &\quad x_{8,2}= w_8 \otimes w_2^\perp ,\\
& \dots \dots \dots \dots \dots &\dots \dots \;  && \quad \dots \dots \dots \dots \dots \\
&x_{1,2m-1}=w_1 \otimes w_{2m-1}^\perp , &\dots \dots \; & &\quad x_{8,2m-1}= w_8 \otimes w_{2m-1}^\perp , \\
&x_{1,2m}=w_1\otimes w_{2m}^\perp  ,&\dots \dots \; & &\quad x_{8,2m}= w_8 \otimes w_{2m}^\perp .
\end{aligned}
\end{equation}

The listed 8-ples of sections can be written formally as octonions, i.e., for $\alpha=1, \dots , 2m$,
\begin{equation}\label{thexx}
\vec x_\alpha= (x_{1,\alpha}, x_{2, \alpha }, \dots , x_{8,\alpha }) = x_{1,\alpha }+ix_{2,\alpha }+ \dots + hx_{8,\alpha},
\end{equation}
and this can be ordered as a $n$-ple of pairs of octonions:
\[
\big( (\vec x_1, \vec x_2), \dots , (\vec x_{2m-1}, \vec x_{2m}) \big) \in (\OO \oplus \OO)^m.
\]

The even Clifford structure on $Gr_{8}(\RR^{2m+8})$ can then be defined by looking at a rank 8 Euclidean vector bundle $E \subset \mathrm{End}^- (T\, Gr_{8}(\RR^{2m+8}))$ that satisfies the condition of being locally generated by anti-commuting orthogonal complex structures. Here, this is denoted by $m_1, m_2, \dots , m_8$, in~correspondence with the  $m_1,m_i,\dots,m_h$. The existence of such an $E$ is insured by the holonomy structure $\mathrm{SO}(2m)\times \mathrm{SO}(8)$ of the Grassmannian, by its spin property, and by the given description of $\Spin{8}$. Accordingly, if $u,v$ are local sections of $E$, we can look at them as octonions in the basis $m_1, m_2, \dots , m_8$. 

For any such orthonormal pair $(u,v)$, look at $u \wedge v$ as a section of $\Lambda^2 E$, and define
\[
\varphi: \Lambda^2 E \rightarrow \mathrm{End}^- (T\, Gr_{8}(\RR^{2m+8}))
\] 
by
\begin{equation}\label{ourphi}
\varphi (u \wedge v) \big( (\vec x_1, \vec x_2), \dots , (\vec x_{2m-1}, \vec x_{2m}) \big) = \big( m_{u,v} (\vec x_1, \vec x_2), \dots , m_{u,v}( \vec x_{2m-1}, \vec x_{2m}) \big),
\end{equation}
i.e., by diagonally applying the matrix \eqref{muv}. When this is extended by the Clifford composition, this gives the Clifford morphism
\[
\varphi: \; \mathrm{Cl}^0 E \rightarrow \End{T\, Gr_{8}(\RR^{2m+8})}.
\] 

Thus, the following theorem applies.

\begin{Theorem} There is a rank 8 vector sub-bundle $E \subset \mathrm{End}^- (T\, Gr_{8}(\RR^{2m+8}))$ that is locally generated by the anti-commuting orthogonal complex structures $m_1, m_2, \dots , m_8$, and $E$ defines an even, non-flat, parallel Clifford structure of rank 8 on $Gr_{8}(\RR^{2m+8})$. The morphism
\[
\varphi: \; \mathrm{Cl}^0 E \rightarrow \mathrm{End} (T\, Gr_{8}(\RR^{2m+8}))
\] 
is given by the Clifford extension of the map,  
\[
u \wedge v \in \Lambda^2 E \longrightarrow [m_{u,v}: (\OO \oplus \OO)^m \rightarrow (\OO \oplus \OO)^m],
\]
which is defined by diagonally applying the matrix \, $m_{u,v}$ \,. Here, $u,v$ are local orthonormal sections of $E$; thus, unitary orthogonal octonions in the basis $m_1, m_2, \dots , m_8$, so that $m_{u,v}$ acts diagonally on the $m$-ples of pairs of local tangent vectors,
\[
\big( (\vec x_1, \vec x_2), \dots , (\vec x_{2m-1}, \vec x_{2m}) \big),
\]
that can be looked at as elements of $(\OO \oplus \OO)^m$.
\end{Theorem}

A similar statement holds for the second series of Grassmannians in \eqref{Gra}, assuming again an~even dimensional ambient space, $\CC^{2m+4}$.   Let  $w_1, w_2, w_3,w_4$ and $w_1^\perp, w_2^\perp, \dots , w_{2m-1}^\perp, w_{2m}^\perp$ be the local orthonormal bases of $W$ and $W^\perp$, respectively. Define the following local tangent vector fields as local sections of $T\, Gr_{4}(\CC^{2m+4}) \cong W \otimes W^\perp$: 
\begin{equation}\label{thez}
\begin{aligned}
&z_{1,1}=w_1 \otimes w_1^\perp , &\dots \dots \;  & &\quad z_{4,1}= w_4 \otimes w_1^\perp ,\\
&z_{1,2}=w_1 \otimes w_2^\perp , &\dots \dots \;  & &\quad z_{4,2}= w_4 \otimes w_2^\perp ,\\
& \dots \dots \dots \dots \dots &\dots \dots \;  && \quad \dots \dots \dots \dots \dots \\
&z_{1,2m-1}=w_1 \otimes w_{2m-1}^\perp , &\dots \dots \; & &\quad z_{4,2m-1}= w_4 \otimes w_{2m-1}^\perp , \\
&z_{1,2m}=w_1\otimes w_{2m}^\perp  ,&\dots \dots \; & &\quad z_{4,2m}= w_4 \otimes w_{2m}^\perp.
\end{aligned}
\end{equation}

Again, look at the above lines as
\begin{equation}\label{thezz}
\vec z_\alpha= (z_{1,\alpha}, z_{2, \alpha }, z_{3, \alpha} z_{4,\alpha }) \in \CC^4,
\end{equation}
($\alpha=1, \dots , 2m$), and order them as $m$-ples of pairs:
\[
\big( (\vec z_1, \vec z_2), \dots , (\vec z_{2n-1}, \vec z_{2n}) \big) \in (\CC^4 \oplus \CC^4)^m.
\]

Now consider the vector sub-space $F=<1,i,j,k,e,f> \subset \OO$, and note that the corresponding operators ($m_u$, with $u \in F$), act on the complex vector space ($\CC^4$). Similarly to what was described for the real Grassmannians, there is a vector sub-bundle, $E^6 \subset \mathrm{End}^- (T \; Gr_4(\CC^{2m+4}))$, that is locally generated by the anti-commuting orthogonal complex structures $m_1,m_2, \dots , m_6$, which corresponds to $m_1,m_i,m_j,m_k,m_e,m_f$. This is due to the holonomy ($\mathrm{S(U}(2m)\times \mathrm{U}(4))$) of the Grassmannian and its spin property. If $(u,v)$ is an orthonormal pair of sections of $E^6$, then $u \wedge v$ is a section of $\Lambda^2 E$, \mbox{and the map}
\[
\varphi: \Lambda^2 E \rightarrow \mathrm{End}^-(T\, Gr_{4}(\CC^{2m+4})),
\] 
given by
\begin{equation}
\varphi (u \wedge v) \big( (\vec z_1, \vec z_2), \dots , (\vec z_{2m-1}, \vec z_{2m}) \big) = \big( m_{u,v} (\vec z_1, \vec z_2), \dots , m_{u,v}( \vec z_{2m-1}, \vec z_{2m}) \big),
\end{equation}
is extended, by Clifford composition, to the Clifford morphism
\[
\varphi: \; \mathrm{Cl}^0 E \rightarrow \End{T\, Gr_{4}(\CC^{2m+4})}.
\] 

Note that the holonomy group ${\mathrm{S(U}(2m)\times \mathrm{U}(4))}$ acts on the model tangent space $\CC^{8m}$, and the orthogonal representation 
\[
\mathrm{S(U}(2m)\times \mathrm{U}(4)) \rightarrow \mathrm{SU}(8m)
\]
defines an equivariant algebra morphism ($\varphi: \mathrm{Cl}^0_6 \rightarrow \End{\CC^{8m}}$ mapping $\mathfrak{su}(4)=\mathfrak{spin}(6) \subset \mathrm{Cl}^0_6$ into $\mathfrak{su}(8m) \subset \End{\CC^{8m}}$). The parallel, non-flat feature of $\varphi$ again follows from the holonomy-based construction. This gives the following theorem.

\begin{Theorem} There is a rank 6 vector sub-bundle, $E \subset \mathrm{End}^- (T\, Gr_{4}(\CC^{2m+4}))$, that is locally generated by the anti-commuting orthogonal complex structures $m_1, m_2, \dots , m_6$, and $E$ defines an even, non-flat, parallel Clifford structure of rank 6 on $Gr_{4}(\CC^{2m+4})$. The morphism
\[
\varphi: \; \mathrm{Cl}^0 E \rightarrow \End{T\, Gr_{4}(\CC^{2m+4})}
\] 
is given by Clifford extension of the map:  
\[
u \wedge v \in \Lambda^2 E \longrightarrow [m_{u,v}: (\CC^4 \oplus \CC^4)^n \rightarrow (\CC^4 \oplus \CC^4)^m],
\]
which is defined by diagonally applying the matrix \; $m_{u,v}$. Here, $u,v$ are local orthonormal sections of $E$, and are thus unitary orthogonal in the basis $m_1, m_2, \dots , m_6$, so that $m_{u,v}$ acts diagonally on the $m$-ples of pairs of local tangent vectors:
\[
\big( (\vec z_1, \vec z_2), \dots , (\vec z_{2m-1}, \vec z_{2m}) \big),
\]
which can be viewed as elements of $(\CC^4 \oplus \CC^4)^m$.
\end{Theorem}

\begin{Remark} When the ambient linear spaces have odd dimensions, similar statements hold, but the Clifford vector bundles $E^8$ and $E^6$ are defined only locally. This fact is due to the spin/non-spin property of the two series of Grassmannians, $Gr_{8}(\RR^{n+8})$ and $Gr_{4}(\CC^{n+4})$, whose second Stiefel Whitney class satisfies $w_2(Gr)=n u$, where $0 \neq u \in H^2(Gr;\ZZ_2)$ (cf. Table \ref{tab4}, where in the non-spin cases the even Clifford structure is referred to as ``projective'').
\end{Remark}

For Grassmannians in the last series, the spin property of $Gr_{2}(\HH^{n+2})$ holds for all values of $n$, due to the vanishing of $H^2(Gr;\ZZ_2)$. The Clifford morphism $\varphi$ is here constructed as follows.  Let~$w_1, w_2$ and $w_1^\perp, \dots , w_{n}^\perp$ be local orthonormal bases of $W$ and $W^\perp$, respectively. Define the following local tangent vector fields as local sections of $T\, Gr_{2}(\HH^{n+2}) \cong W \otimes W^\perp$:
\begin{equation}\label{theh}
\begin{aligned}
&h_{1,1}=w_1 \otimes w_1^\perp , & \quad & h_{2,1}= w_2 \otimes w_1^\perp ,\\
&h_{1,2}=w_1 \otimes w_2^\perp , &\quad &h_{2,2}= w_2 \otimes w_2^\perp ,\\
& \dots \dots \dots \dots \dots & \quad & \dots \dots \dots \dots \dots \\
&h_{1,n-1}=w_1 \otimes w_{n-1}^\perp , &\quad &h_{2,n-1}= w_2 \otimes w_{n-1}^\perp , \\
&h_{1,n}=w_1\otimes w_{n}^\perp  ,&\quad &h_{2,n}= w_2 \otimes w_{2n}^\perp.
\end{aligned}
\end{equation}

Next, let $u,v$ be local orthonormal sections of $E^5$, the sub-bundle of $\mathrm{End}^+ (T\, Gr_{2}(\HH^{n+2}))$ that is locally generated by the Clifford system $C_4$, whose existence is insured by the holonomy of this spin Grassmannian. The composition $uv$ acts diagonally on the $n$-ples:
\[
\big( \vec h_1, \vec h_2, \dots , \vec h_{n-1}, \vec h_{n} \big)=\big ( (h_{1,1}, h_{2,1}), (h_{1,2}, h_{2,2}) , \dots ((h_{1,n}, h_{2,n})  \big) \in (\HH \oplus \HH)^n,
\]
and
\[
\varphi: \; \mathrm{Cl}^0 E \rightarrow \End{T\, Gr_{2}(\HH^{n+2})}
\] 
is given by the Clifford extension of the map
\[
u \wedge v \in \Lambda^2 E \longrightarrow [uv: (\HH \oplus \HH)^{n} \rightarrow  (\HH \oplus \HH)^{n} ].
\]

This gives the following theorem.

\begin{Theorem} There is a rank 5 vector sub-bundle, $E \subset \mathrm{End}^+ (T\, Gr_{2}(\HH^{n+2}))$, that is locally generated by the anti-commuting orthogonal self-dual involutions $\sigma_1, \sigma_2, \dots , \sigma_5$, and $E$ gives rise to an even, non-flat, parallel Clifford structure of rank 5  on $Gr_{2}(\HH^{n+2})$. The Clifford morphism $\varphi$ is constructed as follows. Let $u,v$ be local orthonormal sections of $E$, and let the composition $uv$ act diagonally on the $n$-ples:
\[
\big( \vec h_1, \vec h_2, \dots , \vec h_{n-1}, \vec h_{n} \big),
\]
that can be looked at as elements of $(\HH \oplus \HH)^n$.
Then, the morphism
\[
\varphi: \; \mathrm{Cl}^0 E \rightarrow \End{T\, Gr_{2}(\HH^{n+2})}
\] 
is given by the Clifford extension of the map  
\[
u \wedge v \in \Lambda^2 E \longrightarrow [uv: (\HH \oplus \HH)^{n} \rightarrow  (\HH \oplus \HH)^{n} ].
\]
\end{Theorem}

\begin{Example} Here, we briefly list some properties of the 16-dimensional examples that are included in the just-described even Clifford structures. More details can be found in Ref. \cite{PicSGC}.
From the first series of Grassmannians, one has the "complex octonionic projective line" $$(\CC \otimes \OO) P^1 \cong Gr_8(\RR^{10}) \cong Q_8 \subset \CC P^9,$$ which is totally geodesic in $\EIII$.
There are two parallel even Clifford structures of rank 2 (complex K\"ahler) and of rank 8. Accordingly,
\[
\mathrm{Poin}_{_{Gr_8 (\RR^{10})}} = 1+t^2+t^4+t^6+2t^8+t^{10}+t^{12}+t^{14}+t^{16}.
\]

Next, from the second series, one has the ``third Severi variety'' $$Gr_4(\CC^{6}) \cong V_8^{14} \subset \CC P^{14}.$$
There are three parallel even Clifford structures of rank 2 (complex K\"ahler), rank 3 (quaternion K\"ahler), and rank 6. Here, 
\[
\mathrm{Poin}_{_{Gr_4 (\CC^{6})}} = 1+t^2+2t^4+2t^6+3t^8+2t^{10}+2t^{12}+t^{14}+t^{16}.
\]

Finally, from the third series, one gets $Gr_2(\HH^4),$ with its two families of 2-planes ($\HH P^2$) lying on the Grassmannian and satisfying the classical intersection properties of the Klein quadric, and
\[
\mathrm{Poin}_{_{Gr_2 (\HH^{4})}} = 1+t^4+2t^8+t^{12}+t^{16}.
\]
\end{Example}

\begin{Example} Finally, we mention some higher dimensional examples. First, we address the $32$-dimensional Wolf space $Gr_8(\RR^{12}),$ that has three non-flat, even Clifford structures: two of rank 3, corresponding to the two quaternion K\"ahler structures (in correspondence with two different ways to define hypercomplex structures on the planes on any $Gr_4(\RR^{n+4})$), and one of rank 8, described in this section.
Indeed, $Gr_8(\RR^{12})$ can be looked at as the ``quaternion octonionic'' projective line ($(\HH \otimes \OO)P^1$) that is a total geodesic sub-manifold of the exceptional symmetric space $(\HH \otimes \OO)P^2 \cong \EVI$, cf. \cite{EscRGL}.  Its Poincar\'e polynomial,
\[
\mathrm{Poin}_{_{Gr_4 (\RR^{12})}} = 1+2t^4+4t^8+5t^{12}+6t^{16}+5t^{20}+4t^{24}+2t^{28}+t^{32},
\]
exhibits the presence of two quaternion K\"ahler 4-forms and an ``octonionic K\"ahler'' $8$-form ($\Psi$). The latter is related to one that is defined on $\EVI$ through its holonomy group, $\Spin{12} \cdot \Sp{1}$ (cf. \cite{PicCSE}).

Next, the $64$-dimensional Grassmannian $$Gr_8(\RR^{16})= \frac{\SO{16}}{\SO{8} \times \SO{8}}$$ supports, besides the just-described parallel even Clifford structure of rank 8, another similar structure obtained by interchanging the roles of the two vector bundles $W$ and $W^\perp$, i.e., by operating through the $m_{u,v}$ on elements of $W^\perp$. The real cohomology
\begin{equation}\label{coho}
H^*(Gr_8(\RR^{16})) \cong \frac{\mathbb R [e, p_1, p_2, p_3, e^\perp, p_1^\perp, p_2^\perp, p_3^\perp]}{ee^\perp =0, \; (1+p_1+ p_2+ p_3)(1+p_1^\perp+ p_2^\perp,+p_3^\perp)=1} \;,
\end{equation}
in terms of the Pontrjagin classes $p_\alpha, p_\alpha^\perp$ and Euler classes $e, e^\perp$ of $W$ and $W^\perp$ gives rise to the Poincar\'e~polynomial
\[
\mathrm{Poin}_{_{Gr_8 (\RR^{16})}} = 1+t^4+4t^8+5t^{12}+9t^{16}+11t^{20}+15t^{24}+15t^{28}+18t^{32}+ \dots \; .
\]

These two mentioned even Clifford structures descend to a unique even, parallel Clifford structure of rank 8 on the smooth $\mathbb Z_2$-quotient
\[
Gr^\perp_8(\RR^{16}) = Gr_8(\RR^{16})/\perp
\]
by the orthogonal complement involution $\perp$. The quotient $Gr^\perp_8(\RR^{16})$ turns out to be a totally geodesic, half dimensional sub-manifold of $\EVIII$ and can be viewed as the ``projective line'' ($(\OO \otimes \OO)P^1$) over the ``{octonionic octonions}'' \cite{EscRGL}. For the computation of the cohomology of $Gr^\perp_8(\RR^{16})$, just note that the involution $\perp$ identifies $p_1 \rightarrow p_1^\perp, p_2 \rightarrow p_2^\perp, p_3 \rightarrow p_3^\perp, e \rightarrow e^\perp$. This, due to the relations in \eqref{coho}, allows only the $p_1^2,e,p_1^4, p_1^2 e,p_1^6, p_1^4 e,p_1^8,p_1^6 e$ classes to survive up to dimension 32. This gives the Poincar\'e polynomial
\[
\mathrm{Poin}_{_{Gr^\perp_8 (\RR^{16})}} = 1+2t^8+2t^{16}+2t^{24}+2t^{32}+ \dots \;.
\]
\end{Example}





\begin{sloppypar}
\printbibliography
\end{sloppypar}

\end{document}